\pdfoutput=1

\documentclass[pdflatex,sn-mathphys-num]{sn-jnl}

\usepackage{graphicx}
\usepackage{amsmath,amssymb,amsfonts}
\usepackage{amsthm}
\usepackage{amsopn}
\usepackage{mathrsfs}
\usepackage[title]{appendix}
\usepackage{booktabs}
\usepackage{mathtools}

\theoremstyle{thmstyleone}
\newtheorem{theorem}{Theorem}[section]
\newtheorem{proposition}[theorem]{Proposition}
\newtheorem{lemma}[theorem]{Lemma}
\newtheorem{corollary}[theorem]{Corollary}

\theoremstyle{thmstyletwo}
\newtheorem{remark}[theorem]{Remark}

\theoremstyle{thmstylethree}
\newtheorem{definition}[theorem]{Definition}

\newcommand{\GL}{\mathrm{GL}}

\newcommand{\M}{\mathcal{M}}
\newcommand{\C}{\mathcal{C}}
\newcommand{\E}{\mathcal{E}}
\newcommand{\R}{\mathbb{R}}
\newcommand{\Exp}{\operatorname{Exp}}
\newcommand{\Log}{\operatorname{Log}}
\newcommand{\dist}{\operatorname{dist}}
\newcommand{\diam}{\operatorname{diam}}
\newcommand{\Sec}{\operatorname{Sec}}
\newcommand{\DGUC}{\operatorname{DGUC}}
\newcommand{\LMO}{\operatorname{LMO}}
\newcommand{\SPD}{\mathbb{S}_{++}}
\newcommand{\ip}[2]{\left\langle #1,#2\right\rangle}
\newcommand{\norm}[1]{\left\lVert #1\right\rVert}
\DeclareMathOperator*{\argmin}{arg\,min}
\DeclareMathOperator*{\argmax}{arg\,max}
\DeclareMathOperator{\tr}{tr}
\DeclareMathOperator{\cond}{cond}
\DeclareMathOperator{\diag}{diag}

\begin{document}

\title[Curvature, Scaling, and Complexity]{Curvature, Scaling, and Complexity in Riemannian Projection-Free Optimization}

\author*{\fnm{Binh T.} \sur{Nguyen}}\email{ngtbinh@hcmus.edu.vn}

\affil{\orgdiv{Faculty of Mathematics and Computer Science},
 \orgname{University of Science, Vietnam National University Ho Chi Minh City},
 \city{Ho Chi Minh City}, \country{Vietnam}}

\abstract{
Projection-free methods replace projections by linear minimization oracles,
making their convergence fundamentally sensitive to the geometry of the
feasible region. We develop a geometric framework for this dependence in
Riemannian Frank--Wolfe optimization. A power-type double-geodesic modulus
quantifies the intrinsic geometry of the feasible set, and an exact
transfer principle shows that this modulus induces a Frank--Wolfe scaling
margin modulated by the differential of the Riemannian logarithm map. The
resulting factorization separates the contributions of feasible-set geometry
and ambient Riemannian distortion. Rauch comparison provides explicit
curvature-dependent certificates for this distortion, while on the
positive-definite cone equipped with the affine-invariant metric the
logarithm differential admits an exact spectral characterization in terms
of relative eigenvalue spread. Combined with a quadratic error bound, the
power-type geometry yields the convergence rate
$O(t^{-p/(p-1)})$ for the gap-feedback short-step policy. For every fixed
$p>3$, a compact feasible set in a hyperbolic plane attains the matching
rate $\Theta(t^{-p/(p-1)})$ under the same policy. The geometry-derived
exponent $p/(p-1)$ is therefore sharp for the gap-feedback short-step
policy over the geometry-controlled class considered here.
}



\maketitle

\section{Introduction}
\label{sec:introduction}

Projection-free methods replace potentially expensive projections onto a
feasible region by linear minimization oracles. This distinction is
particularly useful when linear optimization over the constraint set is
substantially cheaper than projection; see more details, for instance,
\cite{combettes2021complexity}. In Euclidean optimization, the geometry of
the feasible region can have a decisive effect on the resulting
complexity. Strongly and uniformly convex feasible sets provide a prominent
example: their boundary geometry induces scaling properties that strengthen
the basic Frank--Wolfe descent mechanism and can yield rates faster than
the classical $O(1/t)$ behavior
\cite{garber2015faster,kerdreux2021projection,kerdreux2026local}. For a
broader perspective on projection-free first-order optimization, we refer
to \cite{bomze2024frankwolfe}.

The same projection-free principle extends naturally to optimization on
Riemannian manifolds. More generally, first-order complexity in
geodesically convex optimization is known to depend on the ambient
geometry, with sectional-curvature bounds entering convergence estimates
through metric comparison and distortion
\cite{zhang2016firstorder,zhang2018estimate,ferreira2019gradient}. Within
the projection-free setting, Weber and Sra
\cite{weber2023riemannian} developed Riemannian Frank--Wolfe (RFW) methods
for constrained Riemannian optimization and established nonasymptotic
guarantees, including applications on the manifold of positive-definite
matrices. More recently, Scieur et al.~\cite{scieur2026strongly}
introduced several notions of strongly convex feasible sets on Riemannian
manifolds, related these notions to Riemannian scaling inequalities, and
used the resulting structure to derive fast RFW convergence.
Chen~\cite{chen2026openloop} subsequently studied fast RFW rates under
local error bounds and directional or exact scaling conditions.

These developments identify scaling inequalities as a central mechanism
behind fast Riemannian Frank--Wolfe convergence. They also expose a
geometric question that is not resolved by assuming a scaling inequality
directly: which intrinsic properties of a Riemannian feasible set generate
such a scaling law, and how does the ambient Riemannian geometry modify
the resulting margin?

There is a well-developed Euclidean analogue of this question. Uniform
convexity of a feasible set can be quantified through a power-type modulus,
and this geometry is closely connected to scaling and error-bound
mechanisms in Frank--Wolfe analysis
\cite{kerdreux2021projection,kerdreux2022restarting,kerdreux2026local}.
Recent lower-bound results have further clarified the corresponding
complexity frontier. In particular, for $p$-uniformly convex Euclidean
sets, Halbey, Roux, and Pokutta~\cite{halbey2026curvature} establish
matching $\Omega(t^{-p/(p-1)})$ behavior for $p\geq 3$ under exact
line-search or short-step policies.

Transferring this picture to a manifold requires more than replacing line
segments by geodesics. A feasible perturbation constructed near the
endpoint of a geodesic is observed from the oracle base point through the
Riemannian logarithm map. Its first-order effect is therefore transformed
by the differential of that map, equivalently by the associated
Jacobi-field geometry. This creates two distinct geometric effects. The
intrinsic geometry of the feasible set determines the size of admissible
transverse perturbations, while the ambient manifold determines how those
perturbations are distorted when viewed from the oracle base point.

The main motivation of this paper is to identify this interaction and use it to
derive projection-free complexity from feasible-set geometry. We introduce
a power-type double-geodesic uniform convexity condition, denoted by
$\DGUC_p(\alpha)$, under which transverse feasible perturbations along a
geodesic have size proportional to
$D^{1-p}d(x,y)^p$, where $d(x,y)$ denotes the Riemannian distance between
$x$ and $y$, and $D$ is the diameter of the feasible set with respect to
this distance. More generally, the geometry is formulated through an
intrinsic double-geodesic modulus, with the power-type condition providing
the specific family used in the complexity analysis.

Our main geometric result is an exact transfer from this feasible-set
modulus to a Frank--Wolfe scaling inequality. Specifically, if $v$ solves
the Riemannian linear minimization oracle at $x$ in direction $w$, then,
under the local regularity assumptions required for the logarithm map,
\begin{equation}
\label{eq:intro-jacobi-scaling}
    \ip{w}{\Log_x v}_x
    \geq
    \alpha D^{1-p}d(x,v)^p
    \norm{(D\Log_x)_v^{*}w}_v .
\end{equation}
The power-type factor in
\eqref{eq:intro-jacobi-scaling} can be replaced by the underlying
feasible-set modulus, so the transfer principle itself is not tied to a
particular exponent.

Inequality~\eqref{eq:intro-jacobi-scaling} makes the two geometric
contributions explicit. The feasible-set modulus determines the intrinsic
scaling margin, while the factor
$
\displaystyle
    \frac{
        \norm{(D\Log_x)_v^{*}w}_v
    }{
        \norm{w}_x
    }
$
measures the ambient differential distortion in the oracle direction. We
refer to this exact geometry-to-scaling transfer as the \emph{Jacobi
scaling principle}. Its proof is local and follows from a one-sided
feasible perturbation at the oracle endpoint together with the first
variation of the logarithm map. In particular, the transfer theorem does
not require a sectional-curvature bound or smoothness of the boundary of
the feasible set.

The exact transfer makes comparison geometry directly applicable to
projection-free scaling. Under a lower sectional-curvature bound, Rauch
comparison provides an explicit lower certificate for the logarithm
differential. In constant sectional curvature $-\kappa^2$, the
corresponding transverse factor is
$
   \displaystyle  \frac{\kappa r}{\sinh(\kappa r)}.
$
Ambient curvature therefore modifies the scaling constant while preserving
the power inherited from the feasible-set modulus. This role of curvature
differs from its appearance in standard first-order Riemannian complexity
estimates \cite{zhang2018estimate,ferreira2019gradient}: here, it quantifies
how feasible-set geometry is transferred into a Frank--Wolfe scaling
margin.

On the manifold of symmetric positive-definite matrices equipped with the
affine-invariant metric, the same mechanism admits an exact spectral
description. If
\[
    X^{-1/2}YX^{-1/2}=\exp(A),
\]
then the singular values of the logarithm differential are determined by
the relative root spreads $|a_i-a_j|/2$. In particular, writing $q$ for
the condition number of $X^{-1/2}YX^{-1/2}$, the smallest singular value
is
\[
    \sigma_{\min}\bigl((D\Log_X)_Y\bigr)
    =
    \frac{\sqrt q\,\log q}{q-1},
\]
with continuous value $1$ at $q=1$. The resulting certificate depends on
relative spectral anisotropy rather than geodesic distance alone and can
therefore be substantially sharper than a global
curvature--diameter certificate.

We next translate the geometry-derived scaling inequality into RFW
complexity. For smooth geodesically convex objectives satisfying a
quadratic error bound, the power-$p$ scaling induced by
$\DGUC_p(\alpha)$ yields
\begin{equation}
\label{eq:intro-upper-rate}
    f(x_t)-f^\star
    =
    O\!\left(t^{-p/(p-1)}\right)
\end{equation}
under the gap-feedback short-step policy. The exponent agrees with the
complexity regime known for uniformly convex feasible sets in Euclidean
space. Here, the scaling inequality is derived from the intrinsic geometry
of the feasible region, with the ambient Riemannian distortion identified
explicitly through the logarithm differential. This provides a geometric
foundation for scaling-based RFW analyses and complements Euclidean
error-bound results such as \cite{kerdreux2022restarting} and recent
Riemannian analyses based directly on scaling and error-bound assumptions
\cite{chen2026openloop}.

We finally establish sharpness of this exponent for the same short-step
policy on a genuinely curved family. For every fixed $p>3$, we construct a
compact geodesically convex feasible set in a hyperbolic plane of
sectional curvature $-\kappa^2$ and consider the boundary-centered
objective
\[
    f_\star(x)
    =
    \frac12 d_\kappa(x,x^\star)^2 .
\]
After normalization, the oracle dynamics form a curvature-dependent
perturbation of the Euclidean slow map. For sufficiently small fixed
$\rho$, the attracting slow state persists under the hyperbolic
perturbation, and the same gap-feedback short-step policy satisfies
\[
    f_\star(x_t)-f_\star(x^\star)
    =
    \Theta\!\left(t^{-p/(p-1)}\right).
\]
Therefore, the geometry-derived upper exponent is attained on a fixed genuinely
negatively curved instance by the same step policy. This establishes
policy-level sharpness of the exponent over the geometry-controlled class
considered here. The corresponding oracle-level complexity question, in
which one ranges over broader classes of Riemannian projection-free
algorithms, remains a distinct problem.

\paragraph{Contributions.}
The main contributions of the paper can be given as follows.
\begin{enumerate}
    \item
    We introduce a power-type double-geodesic uniform convexity condition
    for Riemannian feasible sets and formulate the underlying geometry
    through an intrinsic feasible-set modulus. This separates the
    curvature of the feasible region from that of the ambient manifold.

    \item
    We establish an exact modulus-level geometry-to-scaling transfer. At
    an LMO endpoint, the feasible-set modulus is transported into the
    Frank--Wolfe margin through $(D\Log_x)_v^*$. This gives the Jacobi
    scaling principle in \eqref{eq:intro-jacobi-scaling} and naturally
    leads to an LMO-restricted distortion coefficient that can be sharper
    than a global singular-value certificate.

    \item
    We derive explicit geometric certificates for the logarithmic
    distortion. Rauch comparison gives curvature-dependent lower
    certificates, while on the affine-invariant positive-definite cone, we
    obtain an exact root-space description governed by relative spectral
    spread.

    \item
    We combine the geometry-derived scaling inequality with a quadratic
    error bound to obtain the rate
    $O(t^{-p/(p-1)})$. For every fixed $p>3$, a genuinely negatively
    curved hyperbolic family attains the matching
    $\Theta(t^{-p/(p-1)})$ rate under the same gap-feedback short-step
    policy.
\end{enumerate}

\paragraph{Related work and positioning.}
Our results lie at the intersection of four strands of literature.

\emph{Riemannian first-order and projection-free optimization.}
First-order methods for geodesically convex optimization on Hadamard
manifolds have been analyzed with explicit dependence on sectional
curvature
\cite{zhang2016firstorder,zhang2018estimate,ferreira2019gradient}.
Projection-free Riemannian optimization was subsequently developed by
Weber and Sra, first in stochastic and nonconvex settings and then in a
deterministic RFW framework with sublinear and, under additional
assumptions, linear convergence guarantees
\cite{weber2022stochastic,weber2023riemannian}. More recent work studies
composite Riemannian conditional-gradient methods and alternative
algorithmic regimes \cite{chen2024composite}. These contributions establish
the algorithmic setting in which our results operate. Our focus is on
deriving the scaling margin from geometric properties of the feasible
region.

\emph{Curved feasible sets and Frank--Wolfe scaling.}
In Euclidean spaces, strongly convex feasible sets can yield accelerated
Frank--Wolfe convergence \cite{garber2015faster}. Power-type uniformly
convex sets extend this phenomenon and relate feasible-set moduli to
scaling inequalities and H\"olderian error-bound mechanisms
\cite{kerdreux2021projection,kerdreux2022restarting,kerdreux2026local}.
For an overview of projection-free methods and the role of feasible-set
geometry, one see more details at  \cite{bomze2024frankwolfe}. Uniform convexity also has an
established and distinct meaning for ambient geodesic metric spaces. The
condition studied here instead concerns the geometry of a feasible subset
and is designed to produce a linear-minimization-oracle scaling margin.

The closest geometric precursor is the work of Scieur et
al.~\cite{scieur2026strongly}, who introduce several quadratic notions of
strongly convex sets on Riemannian manifolds and establish implications
toward exact or approximate Riemannian scaling inequalities. Their
double-geodesic condition is particularly close to our starting point. The
present framework extends this perspective to a general feasible-set
modulus, including power laws beyond the quadratic case, and identifies
$(D\Log_x)_v^*$ explicitly as the endpoint transfer operator. This exposes
the ambient Jacobi distortion separately from the intrinsic feasible-set
modulus.

\emph{Error bounds and scaling-based fast rates.}
A complementary line of work begins with an error bound or scaling
inequality and derives faster conditional-gradient convergence. In
Euclidean space, restarting under H\"olderian error bounds provides one
such mechanism \cite{kerdreux2022restarting}. In the Riemannian setting,
Chen~\cite{chen2026openloop} establishes fast rates under local
H\"olderian error bounds and directional or exact scaling conditions,
including open-loop and gap-feedback policies, and gives direct scaling
certificates for selected spherical sets. Our results address the
geometric origin of such scaling inequalities: the margin is derived from
the feasible-set modulus, and its ambient distortion is quantified through
the differential geometry of the logarithm map.

\emph{Lower bounds and sharpness.}
Recent work has substantially refined the Euclidean Frank--Wolfe
lower-bound picture. Halbey et al.~\cite{halbey2026stronglower} give a
constructive $\Omega(1/\sqrt{\varepsilon})$ lower bound even for smooth
strongly convex objectives over strongly convex feasible sets. Grimmer and
Liu~\cite{grimmer2026lmo} establish an oracle lower bound for a broader
deterministic LMO model. For $p$-uniformly convex sets, Halbey, Roux, and
Pokutta~\cite{halbey2026curvature} obtain matching
$\Omega(t^{-p/(p-1)})$ behavior under exact line-search or short-step
policies. These results establish the corresponding Euclidean complexity
frontier. Our hyperbolic construction shows that the same power law
persists after the feasible-set geometry, logarithmic distortion, and
oracle dynamics are placed on a genuinely negatively curved manifold. The
matching result uses the same gap-feedback short-step policy as the upper
bound and establishes policy-level sharpness on the geometry-controlled
class considered here.

Taken together, the existing literature provides the main surrounding
ingredients: Riemannian Frank--Wolfe algorithms, Euclidean uniformly convex
feasible sets, Riemannian strong-set notions, scaling-based convergence
analyses, and Euclidean lower bounds. To the best of our knowledge, these
ingredients have not previously been combined through a modulus-level
transfer in which feasible-set geometry is transported into the
Frank--Wolfe margin by the differential of the logarithm map, with the
resulting distortion quantified by curvature and spectral certificates.

The framework developed here makes this connection explicit. It derives
Frank--Wolfe scaling from intrinsic feasible-set geometry, identifies the
logarithm differential as the exact transfer operator, quantifies its
distortion through comparison and spectral geometry, and connects the
resulting scaling law to projection-free complexity. The hyperbolic
construction further shows that the induced power law is attained under
genuine negative curvature by the same short-step policy. The contribution
is therefore a geometry--scaling--complexity framework rather than an
isolated convergence exponent.

The remainder of the paper is organized as follows.
Section~\ref{sec:preliminaries} introduces the geometric notation, the RFW
oracle, and the scaling and error-bound conditions used in the analysis.
Section~\ref{sec:power-geometry} develops power-type double-geodesic
feasible-set geometry. Section~\ref{sec:jacobi-scaling} establishes the
exact Jacobi scaling principle and its LMO-restricted form.
Section~\ref{sec:certificates} derives curvature and spectral certificates,
including the positive-definite case.
Section~\ref{sec:upper-complexity} develops the geometry-derived RFW upper
bound. Section~\ref{sec:lower-bound} constructs the hyperbolic hard family
and establishes policy-level sharpness for $p>3$.
Section~\ref{sec:numerics} provides numerical illustrations of the rate and
the geometric certificates. Section~\ref{sec:discussion} discusses the
resulting geometry--scaling--complexity framework and directions for
further work.

\section{Preliminaries and problem setting}
\label{sec:preliminaries}

We consider the constrained optimization problem
\begin{equation}
\label{eq:problem}
    \min_{x\in\C} f(x),
\end{equation}
where $(\M,g)$ is a finite-dimensional Riemannian manifold,
$\C\subset\M$ is a feasible set with finite positive diameter
\[
    D:=\diam(\C)\in(0,\infty),
\]
and $f:\C\to\R$ is differentiable. We denote by $d$ the Riemannian
distance, by $T_x\M$ the tangent space at $x$, by
$\ip{\cdot}{\cdot}_x$ the Riemannian inner product, and by
$\norm{\cdot}_x$ the induced norm. The exponential and logarithm maps are
denoted by $\Exp_x$ and $\Log_x$, respectively.

The geometric assumptions used below are local unless explicitly stated
otherwise. Whenever $\Log_x y$ is invoked, $x$ and $y$ are assumed to be
joined by the specified minimizing geodesic and to lie in a normal
configuration on which the logarithm is single-valued. Whenever
$(D\Log_x)_y$ is invoked, we additionally assume that $\Log_x$ is
differentiable at $y$. Completeness, sectional-curvature bounds, or other
global assumptions are imposed only in the results in which they are needed.

\subsection{Geodesics, logarithms, and differential distortion}
\label{subsec:riemannian-prelim}

For $x,y\in\M$ joined by a unique minimizing geodesic, we use the
constant-speed parametrization
\begin{equation}
\label{eq:geodesic-xy}
    \gamma_{xy}(t)
    :=
    \Exp_x\!\bigl(t\Log_x y\bigr),
    \qquad t\in[0,1].
\end{equation}
Thus,
\[
    \gamma_{xy}(0)=x,
    \qquad
    \gamma_{xy}(1)=y,
    \qquad
    d(x,y)=\norm{\Log_x y}_x.
\]
Whenever $\Log_x$ is differentiable at $y$, its differential is the linear
map
\[
    (D\Log_x)_y:T_y\M\longrightarrow T_x\M.
\]
Its metric adjoint
\[
    (D\Log_x)_y^*:T_x\M\longrightarrow T_y\M
\]
is characterized by
\begin{equation}
\label{eq:adjoint-def}
    \ip{w}{(D\Log_x)_y z}_x
    =
    \ip{(D\Log_x)_y^*w}{z}_y,
    \qquad
    w\in T_x\M,\quad z\in T_y\M.
\end{equation}
This adjoint is the natural transfer operator that will relate feasible
endpoint perturbations to the Frank--Wolfe scaling margin.
One direction is unaffected by this differential distortion. Let
$\xi=\Log_x y$ and
$
\displaystyle    \gamma(t)=\Exp_x(t\xi).
$
Since
$
    \Log_x(\gamma(t))=t\xi
$
along the corresponding normal geodesic, differentiation at $t=1$ yields
the radial identity
\begin{equation}
\label{eq:radial-log-identity}
    (D\Log_x)_y\dot\gamma(1)=\xi.
\end{equation}
The behavior in directions transverse to $\dot\gamma(1)$ is instead
governed by Jacobi fields and therefore reflects the ambient Riemannian
geometry. Section~\ref{sec:certificates} quantifies this distortion using
curvature comparison and, for the affine-invariant positive-definite
geometry, spectral information.

\subsection{The Riemannian Frank--Wolfe oracle}
\label{subsec:rfw}

Following the Riemannian Frank--Wolfe framework
\cite{weber2023riemannian}, for $x\in\C$ and $w\in T_x\M$, we define the
logarithmic linear minimization oracle by
\begin{equation}
\label{eq:lmo-general}
    \operatorname{LMO}_{\C}(x,w)
    :=
    \argmax_{v\in\C}
    \ip{w}{\Log_x v}_x.
\end{equation}
The maximization convention in \eqref{eq:lmo-general} is convenient for the
geometric statements below. In the optimization problem
\eqref{eq:problem}, we take
\[
    w=-\nabla f(x),
\]
in which case \eqref{eq:lmo-general} is equivalent to
\begin{equation}
\label{eq:lmo-gradient}
    v\in
    \argmin_{z\in\C}
    \ip{\nabla f(x)}{\Log_x z}_x.
\end{equation}
Existence of an oracle solution is assumed whenever
$\operatorname{LMO}_{\C}(x,w)$ is invoked. In particular, compactness of
$\C$ is not required for the local geometric transfer result merely to be
stated. Compactness will be imposed where it provides the existence and
uniformity properties needed in the algorithmic and lower-bound results.
Given $x_t\in\C$ and an oracle point
\[
    v_t\in
    \operatorname{LMO}_{\C}
    \bigl(x_t,-\nabla f(x_t)\bigr),
\]
the RFW update follows the minimizing geodesic from $x_t$ to $v_t$:
\begin{equation}
\label{eq:rfw-update}
    x_{t+1}
    =
    \Exp_{x_t}
    \!\bigl(\eta_t\Log_{x_t}v_t\bigr),
    \qquad
    \eta_t\in[0,1].
\end{equation}
We use the associated Frank--Wolfe gap
\begin{equation}
\label{eq:fw-gap}
    g_t
    :=
    -\ip{\nabla f(x_t)}
             {\Log_{x_t}v_t}_{x_t}
\end{equation}
and oracle radius
\begin{equation}
\label{eq:oracle-radius}
    r_t:=d(x_t,v_t).
\end{equation}
Suppose that $f$ is geodesically convex on the relevant feasible geodesics
and that the solution set
\[
    X^\star
    :=
    \argmin_{x\in\C} f(x)
\]
is nonempty. Writing
\[
    f^\star:=\min_{x\in\C}f(x),
\]
the first-order inequality for geodesic convexity gives
\begin{equation}
\label{eq:gap-dominates-subopt}
    g_t
    \geq
    f(x_t)-f^\star.
\end{equation}
Indeed, for any $x^\star\in X^\star$,
\[
    f(x^\star)
    \geq
    f(x_t)
    +
    \ip{\nabla f(x_t)}
       {\Log_{x_t}x^\star}_{x_t},
\]
whereas optimality of $v_t$ in \eqref{eq:lmo-gradient} gives
\[
    \ip{\nabla f(x_t)}
       {\Log_{x_t}v_t}_{x_t}
    \leq
    \ip{\nabla f(x_t)}
       {\Log_{x_t}x^\star}_{x_t}.
\]

\subsection{Smoothness, error bounds, and scaling}
\label{subsec:analytic-conditions}

We require smoothness only along the feasible geodesic segments used by the
algorithm. Specifically, we say that $f$ is geodesically $L$-smooth on
$\C$ if
\begin{equation}
\label{eq:geodesic-smoothness}
    f(\gamma_{xy}(\eta))
    \leq
    f(x)
    +
    \eta
    \ip{\nabla f(x)}{\Log_x y}_x
    +
    \frac{L}{2}\eta^2 d(x,y)^2
\end{equation}
for every relevant $x,y\in\C$ and $\eta\in[0,1]$. Applied with
$y=v_t$ to the RFW update \eqref{eq:rfw-update}, this gives the basic
descent estimate
\begin{equation}
\label{eq:master-descent-prelim}
    h_{t+1}
    \leq
    h_t
    -
    \eta_t g_t
    +
    \frac{L}{2}\eta_t^2 r_t^2,
\end{equation}
where
\[
    h_t:=f(x_t)-f^\star.
\]
The objective-side regularity condition used in the rate analysis is a
quadratic error bound. For a nonempty solution set $X^\star$, let
\[
    d(x,X^\star)
    :=
    \inf_{x^\star\in X^\star}d(x,x^\star).
\]
We say that $f$ satisfies a local quadratic error bound on a neighborhood
$U$ of $X^\star$ if there exists $\mu>0$ such that
\begin{equation}
\label{eq:qeb}
    d(x,X^\star)
    \leq
    \mu\bigl(f(x)-f^\star\bigr)^{1/2},
    \qquad
    x\in\C\cap U.
\end{equation}
Strong geodesic convexity provides one sufficient condition for an estimate
of the form \eqref{eq:qeb}, with an appropriate constant, but is not assumed
in the upper-complexity theorem. We retain the error bound as the primitive
objective-side hypothesis in order to separate objective geometry from the
geometry of the feasible set.

The second ingredient is a power-type oracle scaling inequality. Let
$p\geq2$ and $\beta>0$. We say that an oracle triple $(x,w,v)$ satisfies a
$(\beta,p)$ scaling inequality if
\begin{equation}
\label{eq:p-scaling-prelim}
    \ip{w}{\Log_x v}_x
    \geq
    \beta
    \norm{w}_x
    D^{1-p}d(x,v)^p.
\end{equation}
The normalization by $D^{1-p}$ makes $\beta$ invariant under a uniform
rescaling of the metric and is consistent with the power-type feasible-set
modulus introduced in Section~\ref{sec:power-geometry}.

Scaling conditions together with objective-side error bounds are known to
yield accelerated projection-free rates
\cite{scieur2026strongly,chen2026openloop}.
Here, however, \eqref{eq:p-scaling-prelim} is not imposed as an independent
assumption on the optimization problem. Sections~\ref{sec:power-geometry}
--\ref{sec:certificates} derive such a scaling inequality from the geometry
of $\C$ and the differential distortion of the logarithm map.

\subsection{LMO-restricted logarithmic distortion}
\label{subsec:lmo-distortion}

The exact transfer derived later depends only on logarithmic distortions
encountered along oracle-selected directions. To capture this distinction,
define the set of regular, nontrivial oracle configurations by
\begin{equation}
\label{eq:oracle-configurations}
    \E_{\C}
    :=
    \left\{
        (x,w,v):
        \begin{array}{l}
        x,v\in\C,\quad v\neq x,\quad
        0\neq w\in T_x\M,\\[1mm]
        v\in\operatorname{LMO}_{\C}(x,w),\quad
        (D\Log_x)_v\ \text{exists}
        \end{array}
    \right\}.
\end{equation}
The \emph{LMO-restricted logarithmic distortion} is
\begin{equation}
\label{eq:lambda-lmo}
    \Lambda_{\C}^{\LMO}
    :=
    \inf_{(x,w,v)\in\E_{\C}}
    \frac{
        \norm{(D\Log_x)_v^*w}_v
    }{
        \norm{w}_x
    }.
\end{equation}
This is a directional quantity: it probes only endpoint pairs and tangent
directions that can actually arise from the oracle. It may therefore be
strictly more informative than a uniform lower singular-value bound.

For comparison, whenever $(D\Log_x)_v$ exists for every pair in the
corresponding domain, define
\begin{equation}
\label{eq:lambda-global}
    \Lambda_{\C}
    :=
    \inf_{\substack{x,v\in\C\\x\neq v}}
    \sigma_{\min}\!\left((D\Log_x)_v\right).
\end{equation}
For every $(x,w,v)\in\E_{\C}$,
\[
    \norm{(D\Log_x)_v^*w}_v
    \geq
    \sigma_{\min}\!\left((D\Log_x)_v\right)
    \norm{w}_x.
\]
Consequently, whenever the global quantity
\eqref{eq:lambda-global} is defined on a domain containing all
oracle configurations,
\begin{equation}
\label{eq:lambda-order}
    \Lambda_{\C}^{\LMO}
    \geq
    \Lambda_{\C}.
\end{equation}
The distinction between
$\Lambda_{\C}^{\LMO}$ and $\Lambda_{\C}$ will be useful below.
The exact transfer theorem naturally produces the former, whereas curvature
comparison gives convenient lower bounds on the latter. In structured
geometries, such as the affine-invariant positive-definite cone, spectral
information can provide sharper certificates by resolving the distortion
direction by direction.

\section{Power-type curvature of Riemannian feasible sets}
\label{sec:power-geometry}

We now introduce the geometric property of the feasible set that will
generate the Frank--Wolfe scaling inequality. The construction is motivated
by power-type moduli for uniformly convex sets in Euclidean spaces
\cite{kerdreux2026local}, but is formulated intrinsically on the manifold:
the admissible perturbation is taken in the tangent space at an intermediate
point of the minimizing geodesic joining two feasible points. For quadratic
power, the resulting condition is closely related to the double-geodesic
strong-convexity notion of Scieur et al.~\cite{scieur2026strongly}. The
formulation below is chosen so that the feasible-set modulus and the ambient
Riemannian distortion remain separate.

\subsection{A double-geodesic modulus}
\label{subsec:dg-modulus}

Throughout this section, we assume that $\C$ is uniquely geodesically convex:
for every $x,y\in\C$, there is a unique minimizing geodesic
$\gamma_{xy}:[0,1]\to\M$ joining $x$ to $y$, and
$
  \displaystyle  \gamma_{xy}([0,1])\subseteq\C.
$
This assumption makes the global modulus below independent of a choice of
minimizing geodesic. The transfer argument in
Section~\ref{sec:jacobi-scaling} is local and requires substantially less:
only the relevant geodesic terminating at an oracle point and the associated
endpoint perturbations are needed. We use unique geodesic convexity here to
define a convenient global problem class.

\begin{definition}[Double-geodesic modulus]
\label{def:dg-modulus}
Let $\C\subset\M$ be uniquely geodesically convex and let
$D=\diam(\C)\in(0,\infty)$. A function
\[
    \Delta_{\C}:[0,D]\longrightarrow[0,\infty)
\]
is called an admissible \emph{double-geodesic modulus} for $\C$ if, for
every $x,y\in\C$ and $t\in[0,1]$, the tangent ball
\begin{equation}
\label{eq:modulus-tangent-ball}
    \left\{
        z\in T_{\gamma_{xy}(t)}\M:
        \norm{z}_{\gamma_{xy}(t)}
        \leq
        t(1-t)\Delta_{\C}(d(x,y))
    \right\}
\end{equation}
is contained in the domain of
$\Exp_{\gamma_{xy}(t)}$ and satisfies
\begin{equation}
\label{eq:modulus-feasible}
    \Exp_{\gamma_{xy}(t)}(z)\in\C
\end{equation}
for every $z$ in \eqref{eq:modulus-tangent-ball}.
\end{definition}
Equivalently, the geodesic ball generated by exponential perturbations of
radius
$
     \displaystyle   t(1-t)\Delta_{\C}(d(x,y))
$
around $\gamma_{xy}(t)$ is required to remain feasible. The factor
$t(1-t)$ is symmetric under interchange of the endpoints and forces the
admissible perturbation radius to vanish at $t=0$ and $t=1$.

The modulus $\Delta_{\C}(r)$ has the dimension of length. This observation is
important when passing from quadratic to general power-type geometry. In
particular, a condition of the form $\Delta_{\C}(r)\geq\alpha r^p$ assigns
a dimensional parameter to $\alpha$ whenever $p\neq1$. The diameter
normalization introduced below removes this dependence on the choice of
length scale.



\begin{remark}[Admissible moduli and localization]
\label{rem:modulus-admissibility-localization}
The modulus in Definition~\ref{def:dg-modulus} need not be unique. If
$\Delta_{\C}$ is admissible and
$0\leq\widetilde\Delta_{\C}\leq\Delta_{\C}$ pointwise, then
$\widetilde\Delta_{\C}$ is also admissible. Thus the analysis requires only
a specified admissible lower bound; it does not depend on the existence or
regularity of a maximal modulus obtained by taking the pointwise supremum
over admissible radii.

Definition~\ref{def:dg-modulus} is stated globally over pairs of points in
$\C$, whereas the Jacobi scaling argument is local to an oracle
configuration $(x,w,v)$. It uses only feasible perturbations along the
geodesic from $x$ to $v$ and for parameters approaching the endpoint
$t=1$. Accordingly, local, exposed, or oracle-restricted admissible moduli
yield corresponding local scaling statements. We use the global
formulation because it provides a convenient uniform condition for the
complexity results.
\end{remark}

\subsection{Power-type double-geodesic uniform convexity}
\label{subsec:dguc}

\begin{definition}[Power-type double-geodesic uniform convexity]
\label{def:dguc}
Let $p\geq2$ and $\alpha>0$. A uniquely geodesically convex set
$\C\subset\M$ with diameter $D\in(0,\infty)$ is
\emph{$(\alpha,p)$ double-geodesically uniformly convex}, written
\[
    \C\in\DGUC_p(\alpha),
\]
if it admits a double-geodesic modulus $\Delta_{\C}$ satisfying
\begin{equation}
\label{eq:power-modulus}
    \Delta_{\C}(r)
    \geq
    \alpha D^{1-p}r^p,
    \qquad
    0\leq r\leq D.
\end{equation}
Equivalently, for every $x,y\in\C$ and $t\in[0,1]$, every tangent vector
$z\in T_{\gamma_{xy}(t)}\M$ satisfying
\begin{equation}
\label{eq:dguc-radius}
    \norm{z}_{\gamma_{xy}(t)}
    \leq
    \alpha t(1-t)D^{1-p}d(x,y)^p
\end{equation}
belongs to the domain of $\Exp_{\gamma_{xy}(t)}$ and satisfies
$
    \displaystyle    \Exp_{\gamma_{xy}(t)}(z)\in\C.
$
\end{definition}

The factor $D^{1-p}$ makes $\alpha$ dimensionless and is essential for
metric-scale invariance. Indeed, under a constant rescaling
\[
    \widetilde g=c^2g,
    \qquad c>0,
\]
distances, diameters, and tangent-vector norms are all multiplied by $c$,
whereas
\[
    \widetilde D^{\,1-p}\widetilde d(x,y)^p
    =
    (cD)^{1-p}\bigl(c\,d(x,y)\bigr)^p
    =
    cD^{1-p}d(x,y)^p.
\]
Thus, both sides of \eqref{eq:dguc-radius} transform with the same length
scale.

\begin{proposition}[Metric-scale invariance]
\label{prop:scale-invariance}
Let $\widetilde g=c^2g$ for some $c>0$. Then,
\[
    \C\in\DGUC_p(\alpha)
    \quad\text{with respect to }g,
\]
if and only if,
\[
    \C\in\DGUC_p(\alpha)
    \quad\text{with respect to }\widetilde g.
\]
In particular, the parameter $\alpha$ is invariant under constant rescaling
of the Riemannian metric.
\end{proposition}

\begin{proof}
A constant rescaling of the metric leaves the Levi--Civita connection
unchanged. Hence, the parametrized geodesics and exponential maps are the same
as maps on $\M$. On the other hand,
\[
    \widetilde d(x,y)=c\,d(x,y),
    \qquad
    \widetilde D=cD,
    \qquad
    \norm{z}_{\widetilde g}=c\norm{z}_g.
\]
Consequently,
\[
    \widetilde D^{\,1-p}\widetilde d(x,y)^p
    =
    cD^{1-p}d(x,y)^p.
\]
The admissibility inequality \eqref{eq:dguc-radius} is therefore multiplied
by the same factor $c$ on both sides, while its exponential image is
unchanged. The two $\DGUC_p(\alpha)$ conditions are equivalent.
\end{proof}

\subsection{Quadratic case and relation to strong feasible-set geometry}
\label{subsec:p2-relation}

For $p=2$, the radius in Definition~\ref{def:dguc} becomes
\begin{equation}
\label{eq:dguc-p2}
    \alpha t(1-t)D^{-1}d(x,y)^2.
\end{equation}
Therefore, the quadratic case has the same characteristic
$t(1-t)d(x,y)^2$ dependence as double-geodesic strong convexity
\cite{scieur2026strongly}, with the additional diameter normalization that
makes the parameter dimensionless.

The quadratic case $p=2$ recovers the same structural mechanism underlying
the strong-set conditions of \cite{scieur2026strongly}: feasible geodesics
admit transverse perturbations whose guaranteed radius is quadratic in the
separation of their endpoints. Definition~\ref{def:dguc} extends this
mechanism to a power-type modulus. The corresponding constants depend on
the normalization adopted for the feasible-set condition, while the
quadratic geometric structure is preserved.

For $p>2$, the exponent quantifies how the guaranteed transverse feasible
radius scales as two feasible points approach one another. The transfer
theorem in Section~\ref{sec:jacobi-scaling} preserves this power exactly.
The feasible-set geometry contributes the factor
$D^{1-p}d(x,v)^p$, while the ambient manifold enters separately through
the differential factor $(D\Log_x)_v^*$. As a result, the power governing the
Frank--Wolfe scaling law is inherited directly from the intrinsic
feasible-set modulus, whereas ambient Riemannian geometry determines its
distortion. This separation motivates formulating the feasible-set
condition independently of ambient sectional curvature.

\subsection{Monotonicity of the power hierarchy}
\label{subsec:power-monotonicity}

On a set of finite diameter, the normalization in
Definition~\ref{def:dguc} induces a simple monotonicity with respect to the
power.

\begin{proposition}[Power monotonicity]
\label{prop:power-monotonicity}
Suppose that
\[
    \C\in\DGUC_p(\alpha)
\]
for some $p\geq2$. Then, for every $q\geq p$,
\[
    \C\in\DGUC_q(\alpha).
\]
\end{proposition}

\begin{proof}
For every $r\in[0,D]$ and $q\geq p$,
\[
    D^{1-q}r^q
    =
    D^{1-p}r^p
    \left(\frac{r}{D}\right)^{q-p}
    \leq
    D^{1-p}r^p.
\]
Hence,
\[
    \alpha D^{1-q}r^q
    \leq
    \alpha D^{1-p}r^p.
\]
Every perturbation required by $\DGUC_q(\alpha)$ is therefore contained in
the perturbation region guaranteed by $\DGUC_p(\alpha)$.
\end{proof}

\begin{remark}[Effective power]
\label{rem:meaningful-exponent}
By Proposition~\ref{prop:power-monotonicity}, if a feasible set belongs to
$\DGUC_p(\alpha)$, then it also satisfies a power-type condition for every
$q\geq p$, with the corresponding constant. Thus the exponent is most
informative through the small-scale behavior of the modulus
$\Delta_{\C}(r)$ as $r\downarrow0$. When
\[
    \Delta_{\C}(r)\asymp D^{1-p}r^p
    \qquad\text{as } r\downarrow0,
\]
the smallest admissible local exponent $p$ captures the strongest
power-type geometry of the feasible set. This is the effective power used
in the complexity results below; in oracle-restricted statements, the same
interpretation applies to the portion of the feasible set exposed by the
linear minimization oracle.
\end{remark}

\subsection{Separation of feasible-set and ambient geometry}
\label{subsec:dguc-assumptions}

Definition~\ref{def:dguc} isolates the intrinsic geometry of the feasible
set through a purely metric--geometric inclusion property. It applies
without boundary smoothness: the admissible transverse perturbations are
specified directly through the exponential map, independently of a
differentiable boundary or a second fundamental form.

The definition is likewise independent of ambient sectional curvature.
Curvature bounds enter the analysis at a later stage, where they provide
quantitative control of the differential of the logarithm map. This
separation allows the intrinsic curvature of the feasible region and the
ambient Riemannian distortion to be analyzed as distinct geometric
quantities.

The local geometry-to-scaling transfer also extends beyond compact feasible
sets. Finite diameter provides the normalization of the power-type modulus,
while compactness is invoked separately when global properties are needed,
such as existence of the linear minimization oracle or uniform constants in
the algorithmic and sharpness results.

This formulation reveals the two geometric quantities governing the
Frank--Wolfe margin. The feasible-set modulus
$\Delta_{\C}(d(x,v))$ measures the strength of the admissible perturbation
at the oracle endpoint, while
$\norm{(D\Log_x)_v^*w}_v/\norm{w}_x$ measures its differential distortion
when viewed from the oracle base point. The Jacobi scaling theorem in the
next section shows that their product determines the resulting
Frank--Wolfe scaling margin.

\section{From feasible-set curvature to Frank--Wolfe scaling}
\label{sec:jacobi-scaling}

We now establish the geometric transfer principle underlying the subsequent
complexity analysis. The result is stated first for a general
double-geodesic modulus and then specialized to the power-type condition of
Definition~\ref{def:dguc}. In this form, the transfer is independent of a
particular power law: the feasible-set modulus enters without approximation,
while the ambient Riemannian geometry appears through the differential of
the logarithm map.

\subsection{Endpoint first variation}
\label{subsec:endpoint-variation}

The argument is based on a feasible perturbation whose base point approaches
the endpoint of the oracle geodesic. We first isolate the corresponding
first-variation identity.

\begin{lemma}[Endpoint perturbation]
\label{lem:endpoint-perturbation}
Let $x,v\in\M$ be joined by a minimizing geodesic
\[
    \gamma(t)=\Exp_x(t\xi),
    \qquad
    \xi=\Log_x v,
    \qquad
    t\in[0,1].
\]
Let $Z$ be a smooth vector field along $\gamma$, and let $a\geq0$. For
sufficiently small $\varepsilon\geq0$, suppose that
\begin{equation}
\label{eq:q-epsilon}
    q_\varepsilon
    :=
    \Exp_{\gamma(1-\varepsilon)}
    \!\left(
        \varepsilon(1-\varepsilon)aZ(1-\varepsilon)
    \right)
\end{equation}
is defined. Then, $q_0=v$ and
\begin{equation}
\label{eq:q-derivative}
    \left.
    \frac{d}{d\varepsilon}q_\varepsilon
    \right|_{\varepsilon=0^+}
    =
    -\dot\gamma(1)+aZ(1).
\end{equation}
\end{lemma}

\begin{proof}
Let
\[
    \mathscr E:\mathcal D\subset T\M\longrightarrow\M,
    \qquad
    \mathscr E(y,u):=\Exp_y(u),
\]
where $\mathcal D$ denotes the domain of the exponential map. Consider the
curve in the tangent bundle
\[
    \Gamma(\varepsilon)
    :=
    \left(
        \gamma(1-\varepsilon),
        \varepsilon(1-\varepsilon)aZ(1-\varepsilon)
    \right).
\]
Then, $q_\varepsilon=\mathscr E(\Gamma(\varepsilon))$ and
$
  \displaystyle  \Gamma(0)=(v,0).
$
At the zero section, the differential of $\mathscr E$ identifies the
horizontal variation of the base point and the vertical variation of the
fiber with their sum in $T_v\M$. The base-point component of
$\dot\Gamma(0)$ is $-\dot\gamma(1)$. Since the fiber component vanishes at
$\varepsilon=0$, its first-order vertical variation is $aZ(1)$; derivatives
of $Z(1-\varepsilon)$ are multiplied by the vanishing factor
$\varepsilon(1-\varepsilon)$ and do not contribute at first order.
Consequently,
\[
    \dot q_0
    =
    -\dot\gamma(1)+aZ(1),
\]
which proves \eqref{eq:q-derivative}.
\end{proof}

\begin{remark}[Independence of the extension]
\label{rem:no-transport-derivative}
Only the terminal value $Z(1)$ enters
\eqref{eq:q-derivative}. Hence, the first variation is independent of the
particular smooth extension of a prescribed vector
$z\in T_v\M$ along $\gamma$. In the proof below we may, for example, use
parallel transport of $z$ along the oracle geodesic.
\end{remark}

\subsection{Exact transfer of the double-geodesic modulus}
\label{subsec:exact-transfer}

We now state the main geometric result. Its assumptions are local at the
selected oracle configuration apart from the use of the global modulus from
Definition~\ref{def:dg-modulus}. In particular, no sectional-curvature bound
and no regularity assumption on the boundary of $\C$ are required.

\begin{theorem}[Exact Jacobi scaling transfer]
\label{thm:exact-jacobi-transfer}
Let $\C\subset\M$ be uniquely geodesically convex, and let
$\Delta_{\C}$ be an admissible double-geodesic modulus in the sense of
Definition~\ref{def:dg-modulus}. Fix $x\in\C$ and
$0\neq w\in T_x\M$, and let
\begin{equation}
\label{eq:thm-lmo}
    v\in
    \argmax_{z\in\C}
    \ip{w}{\Log_x z}_x.
\end{equation}
Assume that $\Log_x$ is single valued along the relevant oracle geodesic and
differentiable at $v$. Then, with
\[
    r:=d(x,v),
\]
one has
\begin{equation}
\label{eq:exact-modulus-transfer}
    \ip{w}{\Log_x v}_x
    \geq
    \Delta_{\C}(r)
    \norm{(D\Log_x)_v^*w}_v .
\end{equation}
In particular, if $\C\in\DGUC_p(\alpha)$, then
\begin{equation}
\label{eq:exact-power-transfer}
    \ip{w}{\Log_x v}_x
    \geq
    \alpha D^{1-p}d(x,v)^p
    \norm{(D\Log_x)_v^*w}_v .
\end{equation}
\end{theorem}

\begin{proof}
If $v=x$, then $r=0$ and the right-hand side of
\eqref{eq:exact-modulus-transfer} vanishes, so the claim is immediate.
Assume henceforth that $v\neq x$, and set
\[
    \xi:=\Log_x v,
    \qquad
    r=\norm{\xi}_x=d(x,v),
    \qquad
    \gamma(t):=\Exp_x(t\xi).
\]
Let
\[
    a:=\Delta_{\C}(r).
\]
Fix an arbitrary unit vector $z\in T_v\M$. Choose a smooth unit vector field
$Z$ along $\gamma$ such that $Z(1)=z$; parallel transport of $z$ along
$\gamma$ provides one such choice. For $\varepsilon>0$ sufficiently small,
define
\[
    q_\varepsilon
    :=
    \Exp_{\gamma(1-\varepsilon)}
    \!\left(
        \varepsilon(1-\varepsilon)aZ(1-\varepsilon)
    \right).
\]
Since $t=1-\varepsilon$ and
\[
    \left\|
        \varepsilon(1-\varepsilon)aZ(1-\varepsilon)
    \right\|_{\gamma(1-\varepsilon)}
    =
    \varepsilon(1-\varepsilon)\Delta_{\C}(r),
\]
Definition~\ref{def:dg-modulus}, applied to the pair $(x,v)$, gives
$
  \displaystyle  q_\varepsilon\in\C.
$
Define, on a neighborhood of $v$ on which $\Log_x$ is defined,
\[
    F_x(y):=\ip{w}{\Log_x y}_x.
\]
By the oracle property \eqref{eq:thm-lmo},
\[
    F_x(q_\varepsilon)
    \leq
    F_x(v)
    =
    F_x(q_0).
\]
Hence, the feasible one-sided derivative satisfies
\begin{equation}
\label{eq:one-sided-optimality}
    \left.
    \frac{d}{d\varepsilon}
    F_x(q_\varepsilon)
    \right|_{\varepsilon=0^+}
    \leq0.
\end{equation}
Using Lemma~\ref{lem:endpoint-perturbation},
\[
    \dot q_0
    =
    -\dot\gamma(1)+az.
\]
The chain rule and the definition of the metric adjoint therefore give
\begin{equation}
\label{eq:adjoint-first-variation}
    \ip{
        (D\Log_x)_v^*w
    }{
        -\dot\gamma(1)+az
    }_v
    \leq0.
\end{equation}
The radial identity \eqref{eq:radial-log-identity} yields
\[
    (D\Log_x)_v\dot\gamma(1)
    =
    \xi,
\]
and hence,
\begin{equation}
\label{eq:radial-adjoint-identity}
    \ip{(D\Log_x)_v^*w}{\dot\gamma(1)}_v
    =
    \ip{w}{\xi}_x.
\end{equation}
Substituting \eqref{eq:radial-adjoint-identity} into
\eqref{eq:adjoint-first-variation} gives
\begin{equation}
\label{eq:directional-transfer}
    \ip{w}{\xi}_x
    \geq
    a
    \ip{(D\Log_x)_v^*w}{z}_v.
\end{equation}
Since \eqref{eq:directional-transfer} holds for every unit
$z\in T_v\M$, taking the supremum over the unit sphere gives
\[
    \ip{w}{\xi}_x
    \geq
    a
    \norm{(D\Log_x)_v^*w}_v.
\]
Recalling that $a=\Delta_{\C}(r)$ and
$\xi=\Log_xv$ proves \eqref{eq:exact-modulus-transfer}.

If $\C\in\DGUC_p(\alpha)$, then
\[
    \Delta_{\C}(r)
    \geq
    \alpha D^{1-p}r^p,
\]
and \eqref{eq:exact-power-transfer} follows immediately.
\end{proof}

\begin{remark}[Local and oracle-wise transfer]
\label{rem:local-oracle-transfer}
The exact transfer is local both geometrically and with respect to the
linear minimization oracle. Its proof uses the feasible-set modulus only
along the selected oracle geodesic $\gamma_{xv}$ and for parameters
$t\uparrow1$. Hence, Theorem~\ref{thm:exact-jacobi-transfer} applies
whenever a local, exposed, or oracle-restricted modulus supplies the
corresponding endpoint perturbations. The global
$\DGUC_p(\alpha)$ condition subsequently provides a uniform certificate
over the problem class.

The same pointwise character applies when the linear minimization oracle
has multiple solutions. For every maximizer $v$ satisfying
\eqref{eq:thm-lmo} at which the logarithm map has the stated local
regularity, the transfer inequality holds with the corresponding
$\Log_x v$ and $(D\Log_x)_v^*$. Thus, different oracle solutions may carry
different directional distortion factors, a distinction captured
naturally by the LMO-restricted coefficients introduced below. The
normalized coefficients are stated for $w\neq0$, while the unnormalized
transfer inequality extends directly to the zero direction.
\end{remark}

\subsection{Scalar scaling and LMO-restricted distortion}
\label{subsec:scalarization}

The exact transfer
\eqref{eq:exact-modulus-transfer} retains the direction-dependent
differential factor
\[
    \norm{(D\Log_x)_v^*w}_v.
\]
A scalar Frank--Wolfe scaling inequality is obtained by lower-bounding this
factor relative to $\norm{w}_x$.

\begin{corollary}[LMO-restricted power scaling]
\label{cor:lmo-power-scaling}
Suppose that $\C\in\DGUC_p(\alpha)$ and that
$\Lambda_{\C}^{\LMO}$ is defined by \eqref{eq:lambda-lmo}. Then, every
$(x,w,v)\in\E_{\C}$ satisfies
\begin{equation}
\label{eq:lmo-scalar-scaling}
    \ip{w}{\Log_x v}_x
    \geq
    \alpha\Lambda_{\C}^{\LMO}
    \norm{w}_x
    D^{1-p}d(x,v)^p.
\end{equation}
Consequently, if
\[
    \Lambda_{\C}^{\LMO}>0,
\]
then the oracle configurations satisfy the $(\beta,p)$ scaling inequality
of \eqref{eq:p-scaling-prelim} with
\begin{equation}
\label{eq:beta-alpha-lambda}
    \beta
    =
    \alpha\Lambda_{\C}^{\LMO}.
\end{equation}
\end{corollary}

\begin{proof}
For every $(x,w,v)\in\E_{\C}$,
Definition~\eqref{eq:lambda-lmo} gives
\[
    \norm{(D\Log_x)_v^*w}_v
    \geq
    \Lambda_{\C}^{\LMO}\norm{w}_x.
\]
Combining this inequality with
\eqref{eq:exact-power-transfer} yields
\eqref{eq:lmo-scalar-scaling}.
\end{proof}

The preceding corollary uses only distortions encountered by the oracle.
A more conservative but often easier certificate is obtained from a uniform
singular-value bound.

\begin{corollary}[Uniform singular-value certificate]
\label{cor:singular-value-scaling}
Suppose that $(D\Log_x)_v$ exists on a domain containing all relevant oracle
pairs and that
\[
    \Lambda_{\C}
    :=
    \inf_{\substack{x,v\in\C\\x\neq v}}
    \sigma_{\min}\!\left((D\Log_x)_v\right)
    >0,
\]
where the infimum is taken over pairs for which the logarithm differential
is defined. Then,
\[
    \Lambda_{\C}^{\LMO}
    \geq
    \Lambda_{\C},
\]
and every oracle configuration satisfies
\begin{equation}
\label{eq:global-scalar-scaling}
    \ip{w}{\Log_x v}_x
    \geq
    \alpha\Lambda_{\C}
    \norm{w}_x
    D^{1-p}d(x,v)^p.
\end{equation}
\end{corollary}

\begin{proof}
For every relevant pair $(x,v)$ and every $w\in T_x\M$,
the linear map $(D\Log_x)_v$ and its metric adjoint have the same singular
values. Therefore
\[
    \norm{(D\Log_x)_v^*w}_v
    \geq
    \sigma_{\min}\!\left((D\Log_x)_v\right)
    \norm{w}_x.
\]
Taking the infimum over oracle configurations yields
\[
    \Lambda_{\C}^{\LMO}
    \geq
    \Lambda_{\C}.
\]
The claimed scaling inequality then follows from
Corollary~\ref{cor:lmo-power-scaling}.
\end{proof}

\begin{remark}[Oracle restriction versus uniform distortion]
\label{rem:directional-sharper}
The distinction between
$\Lambda_{\C}^{\LMO}$ and $\Lambda_{\C}$ is structural.
The global quantity $\Lambda_{\C}$ controls every tangent direction for
every admissible pair of feasible points, whereas
$\Lambda_{\C}^{\LMO}$ probes only directions and endpoints selected by the
linear minimization oracle. Hence, the latter may remain bounded away from
zero even when a global singular-value estimate is weak or overly
conservative. Section~\ref{sec:certificates} develops both types of
certificates.
\end{remark}

\subsection{The quadratic case and Jacobi distortion}
\label{subsec:p2-jacobi}

For $p=2$, Theorem~\ref{thm:exact-jacobi-transfer} gives
\begin{equation}
\label{eq:p2-jacobi-transfer}
    \ip{w}{\Log_x v}_x
    \geq
    \alpha D^{-1}d(x,v)^2
    \norm{(D\Log_x)_v^*w}_v.
\end{equation}
This identity clarifies the relation with the quadratic strong-set framework
of Scieur et al.~\cite{scieur2026strongly}. Their double-geodesic
strong-convexity condition leads to an approximate Riemannian scaling
inequality through a double-exponential analysis, whereas their stronger
Riemannian strong-convexity condition yields an exact scaling relation.
The inequality~\eqref{eq:p2-jacobi-transfer} identifies explicitly the
differential factor that arises from endpoint perturbations: double-geodesic
feasible-set curvature gives an exact direction-sensitive inequality, while
a scalar scaling constant is obtained by controlling
$
  \displaystyle  \frac{
        \norm{(D\Log_x)_v^*w}_v
    }{
        \norm{w}_x
    }.
$
In Euclidean space,
\[
    (D\Log_x)_v=I,
\]
after the canonical identification of tangent spaces, and the differential
factor reduces to one. On a curved manifold, transverse directions can be
expanded or contracted by the logarithm differential. Theorem
\ref{thm:exact-jacobi-transfer} retains this effect exactly rather than
absorbing it into an additive approximation term. Thus, the feasible-set
modulus and the ambient differential distortion remain multiplicatively
separated.

\begin{remark}[No curvature hypothesis in the transfer theorem]
\label{rem:no-curvature-transfer}
Sectional-curvature bounds are not required for
Theorem~\ref{thm:exact-jacobi-transfer}. They enter only when the
direction-dependent factor
\[
    \frac{
        \norm{(D\Log_x)_v^*w}_v
    }{
        \norm{w}_x
    }
\]
is replaced by a tractable scalar lower bound. Consequently, the power $p$
is inherited from the feasible-set modulus, whereas ambient curvature
affects the conditioning of the resulting scaling margin rather than its
power. Section~\ref{sec:certificates} makes this distinction quantitative.
\end{remark}

\section{Curvature and spectral certificates}
\label{sec:certificates}

The exact transfer theorem of Section~\ref{sec:jacobi-scaling} retains the
direction-dependent factor
$
  \displaystyle  \frac{\norm{(D\Log_x)_v^*w}_v}{\norm{w}_x}.
$
We now develop two complementary ways of controlling this quantity. The
first uses sectional-curvature comparison and applies on general Riemannian
manifolds. The second exploits the symmetric-space geometry of the
positive-definite cone and gives an exact spectral formula. The former is
portable but generally conservative; the latter retains the anisotropic
spectral information of the particular pair of positive-definite matrices.

\subsection{A Rauch-comparison certificate}
\label{subsec:rauch-certificate}

For $K\in\R$, define the model function
\begin{equation}
\label{eq:model-S-kappa}
    S_K(r)
    :=
    \begin{cases}
        \dfrac{\sin(\sqrt{K}\,r)}{\sqrt{K}},
        & K>0,\\[1.2ex]
        r,
        & K=0,\\[1.2ex]
        \dfrac{\sinh(\sqrt{-K}\,r)}{\sqrt{-K}},
        & K<0.
    \end{cases}
\end{equation}
The following proposition is the form of Rauch comparison needed in the
sequel; see, e.g., standard treatments of Jacobi-field comparison
\cite{petersen2016riemannian}.

\begin{proposition}[Rauch certificate]
\label{prop:rauch-certificate}
Let $x,v\in\M$ be joined by a unique minimizing geodesic of length
\[
    r:=d(x,v)>0,
\]
and assume that $v$ is a regular point of $\Exp_x$ along this geodesic.
Suppose that the sectional curvatures of all two-planes containing the
geodesic direction satisfy
\[
    \Sec\geq K_-
\]
along the segment, for some $K_-\in\R$. When $K_->0$, assume in addition
that
\[
    r<\frac{\pi}{\sqrt{K_-}},
\]
so that the comparison model remains before its first conjugate point.
Then
\begin{equation}
\label{eq:rauch-dlog}
    \sigma_{\min}\!\bigl((D\Log_x)_v\bigr)
    \geq
    \Psi_{K_-}(r)
    :=
    \min\left\{
        1,\frac{r}{S_{K_-}(r)}
    \right\}.
\end{equation}
Consequently, if the same assumptions hold for every pair in a regular
region containing $\C$, then
\begin{equation}
\label{eq:rauch-global-certificate}
    \Lambda_{\C}
    \geq
    \inf_{0<r\leq D}
    \Psi_{K_-}(r).
\end{equation}
\end{proposition}

\begin{proof}
Let
\[
    \xi:=\Log_xv,
    \qquad
    \norm{\xi}_x=r.
\]
The differential of $\Exp_x$ at $\xi$ has singular value one in the radial
direction. For a direction orthogonal to $\xi$, the differential is
represented by a Jacobi field with vanishing initial value and prescribed
initial derivative.
Rauch comparison under the lower sectional-curvature bound
$\Sec\geq K_-$ gives
\[
    \norm{(D\Exp_x)_\xi u}_v
    \leq
    \frac{S_{K_-}(r)}{r}\norm{u}_x
\]
for transverse $u$, on the stated comparison interval. Since $v$ is regular,
$(D\Exp_x)_\xi$ is invertible and
\[
    (D\Log_x)_v
    =
    \bigl((D\Exp_x)_\xi\bigr)^{-1}.
\]
The radial singular value of $(D\Log_x)_v$ is therefore one, while its
transverse singular values are bounded below by
$
  \displaystyle  \frac{r}{S_{K_-}(r)}.
$
Taking the smaller of the radial and transverse bounds proves
\eqref{eq:rauch-dlog}. The uniform estimate
\eqref{eq:rauch-global-certificate} follows by taking the infimum over the
relevant pairs.
\end{proof}

The direction of the curvature comparison is important. A lower
sectional-curvature bound limits the transverse expansion of the exponential
map from above and therefore yields a lower bound on the inverse
differential $D\Log$. By contrast, upper curvature information is relevant
to normal-neighborhood and conjugate-radius control but does not provide the
lower scaling factor in \eqref{eq:rauch-dlog}.

\begin{corollary}[Uniform negative-curvature certificate]
\label{cor:negative-curvature-certificate}
Suppose that
\[
    \Sec\geq-\kappa^2,
    \qquad
    \kappa\geq0,
\]
on the relevant regular normal region. Then, for every relevant pair
$x,v$ with $r=d(x,v)$,
\begin{equation}
\label{eq:hyperbolic-certificate}
    \sigma_{\min}\!\bigl((D\Log_x)_v\bigr)
    \geq
    \frac{\kappa r}{\sinh(\kappa r)},
\end{equation}
where the right-hand side is interpreted as $1$ for $\kappa=0$.
Consequently, if $r\leq D$ for all relevant pairs,
\begin{equation}
\label{eq:hyperbolic-global-certificate}
    \Lambda_{\C}
    \geq
    \frac{\kappa D}{\sinh(\kappa D)}.
\end{equation}
If, in addition, $\C\in\DGUC_p(\alpha)$, then every corresponding LMO
configuration satisfies
\begin{equation}
\label{eq:hyperbolic-prsi}
    \ip{w}{\Log_xv}_x
    \geq
    \alpha
    \frac{\kappa D}{\sinh(\kappa D)}
    D^{1-p}
    \norm{w}_x
    d(x,v)^p.
\end{equation}
\end{corollary}

\begin{proof}
For $K_-=-\kappa^2$,
\[
    S_{K_-}(r)
    =
    \frac{\sinh(\kappa r)}{\kappa},
\]
with the usual continuous interpretation at $\kappa=0$.
Proposition~\ref{prop:rauch-certificate} therefore gives
\eqref{eq:hyperbolic-certificate}. Since
\[
    s\longmapsto\frac{s}{\sinh s}
\]
is decreasing on $[0,\infty)$,
\eqref{eq:hyperbolic-global-certificate} follows from $r\leq D$.
The scaling estimate \eqref{eq:hyperbolic-prsi} then follows from
Corollary~\ref{cor:singular-value-scaling}.
\end{proof}

For small $\kappa D$,
\begin{equation}
\label{eq:hyperbolic-small-radius}
    \frac{\kappa D}{\sinh(\kappa D)}
    =
    1-\frac{\kappa^2D^2}{6}
    +\frac{7\kappa^4D^4}{360}
    +O\!\left((\kappa D)^6\right).
\end{equation}
Thus, on a bounded regular region, negative curvature perturbs the scalar
scaling certificate only at second order in $\kappa D$ near the flat limit,
while leaving the power inherited from the feasible-set modulus unchanged.

\begin{remark}[Nonnegative lower curvature bounds]
\label{rem:nonnegative-curvature}
Suppose that $K_-\geq0$ and that the relevant logarithm maps remain regular
along the comparison interval. Since
\[
    S_{K_-}(r)\leq r,
\]
Proposition~\ref{prop:rauch-certificate} yields
\[
    \sigma_{\min}\!\bigl((D\Log_x)_v\bigr)\geq1.
\]
Hence, throughout the regular comparison region, nonnegative lower
sectional curvature introduces no loss in the lower differential
certificate for the logarithm map. The statement applies up to the
corresponding cut or conjugate threshold.
\end{remark}

\subsection{Constant-curvature sanity checks}
\label{subsec:space-form-checks}

The comparison factor is attained in the corresponding constant-curvature
models. In Euclidean space,
\[
    \sigma\bigl((D\Log_x)_v\bigr)=1
\]
in every direction. On a sphere of constant sectional curvature
$+\kappa^2$, before the first conjugate point, the radial singular value is
one and every transverse singular value equals
$
    \frac{\kappa r}{\sin(\kappa r)}.
$
On hyperbolic space of constant sectional curvature $-\kappa^2$, the radial
singular value remains one and every transverse singular value equals
$
    \frac{\kappa r}{\sinh(\kappa r)}.
$
Thus, positive curvature contracts transverse exponential variations and
negative curvature expands them, with the inverse behavior for the
logarithm differential.

These model spaces also illustrate the limitation of a diameter-only
certificate. In negative curvature,
\begin{equation}
\label{eq:hyperbolic-large-radius}
    \frac{\kappa D}{\sinh(\kappa D)}
    \sim
    2\kappa D e^{-\kappa D},
    \qquad
    \kappa D\to\infty.
\end{equation}
Hence, the worst-case global certificate can become exponentially small even
when the directional distortion encountered along the actual LMO trajectory
is substantially better.

\subsection{An exact spectral certificate on the positive-definite cone}
\label{subsec:spd-certificate}

We next specialize to the manifold $\SPD(n)$ of real symmetric
positive-definite $n\times n$ matrices endowed with the affine-invariant
Riemannian metric
\begin{equation}
\label{eq:airm-metric}
    g_X(U,V)
    :=
    \tr(X^{-1}UX^{-1}V).
\end{equation}
This manifold is complete, simply connected, and of nonpositive sectional
curvature. Congruence transformations
\[
    X\longmapsto GXG^\top,
    \qquad G\in\GL(n),
\]
act by isometries; see, e.g., \cite{bhatia2007positive}. The exponential and
logarithm maps are
\begin{align}
    \Exp_X(U)
    &=
    X^{1/2}
    \exp\!\left(X^{-1/2}UX^{-1/2}\right)
    X^{1/2},
    \label{eq:spd-exp}\\
    \Log_X(Y)
    &=
    X^{1/2}
    \log\!\left(X^{-1/2}YX^{-1/2}\right)
    X^{1/2}.
    \label{eq:spd-log}
\end{align}

A global curvature bound already yields a Rauch certificate. The symmetric
space structure, however, permits the singular spectrum of the logarithm
differential to be computed exactly.

\begin{theorem}[Exact AIRM logarithm distortion]
\label{thm:spd-exact-distortion}
Let $X,Y\in\SPD(n)$ and write
\begin{equation}
\label{eq:spd-relative-spectrum}
    X^{-1/2}YX^{-1/2}
    =
    Q\exp(A)Q^\top,
    \qquad
    A=\diag(a_1,\ldots,a_n).
\end{equation}
Viewed as a linear map between the AIRM inner-product spaces
\[
    (T_Y\SPD(n),g_Y)
    \quad\text{and}\quad
    (T_X\SPD(n),g_X),
\]
the differential $(D\Log_X)_Y$ has singular value $1$ on the commuting
directions and, for each $i<j$, singular value
\begin{equation}
\label{eq:spd-root-singular}
    \phi\!\left(\frac{|a_i-a_j|}{2}\right),
    \qquad
    \phi(s):=\frac{s}{\sinh s},
\end{equation}
on the corresponding off-diagonal root direction. Consequently,
\begin{equation}
\label{eq:spd-sigma-min}
    \sigma_{\min}\!\bigl((D\Log_X)_Y\bigr)
    =
    \phi\!\left(
        \frac12\max_{i,j}|a_i-a_j|
    \right)
    =
    \phi\!\left(
        \frac12
        \log\cond(X^{-1/2}YX^{-1/2})
    \right).
\end{equation}
Equivalently, with
\[
    q:=\cond(X^{-1/2}YX^{-1/2})\geq1,
\]
one has
\begin{equation}
\label{eq:spd-condition-number-factor}
    \sigma_{\min}\!\bigl((D\Log_X)_Y\bigr)
    =
    \frac{\sqrt q\,\log q}{q-1},
\end{equation}
with the continuous value $1$ at $q=1$.
\end{theorem}

\begin{proof}
By congruence invariance, the singular spectrum is unchanged if the pair
$(X,Y)$ is mapped isometrically to
\[
    (I,X^{-1/2}YX^{-1/2}).
\]
After an orthogonal change of basis, it therefore suffices to consider
\[
    X=I,
    \qquad
    Y=e^A,
    \qquad
    A=\diag(a_1,\ldots,a_n).
\]
At $I$, the AIRM coincides with the Frobenius inner product. The differential
of the matrix exponential at $A$ is
\[
    D\exp_A[H]
    =
    \int_0^1
    e^{(1-t)A}He^{tA}\,dt.
\]
For a diagonal matrix unit, this differential multiplies the corresponding
component by $e^{a_i}$. For an off-diagonal symmetric direction associated
with $i<j$, it multiplies the corresponding entry by
\[
    \frac{e^{a_i}-e^{a_j}}{a_i-a_j}
    =
    e^{(a_i+a_j)/2}
    \frac{
        \sinh\!\left((a_i-a_j)/2\right)
    }{
        (a_i-a_j)/2
    },
\]
with the continuous interpretation when $a_i=a_j$.
The AIRM norm at $Y=e^A$ is
\[
    \norm{K}_Y
    =
    \norm{e^{-A/2}Ke^{-A/2}}_F.
\]
Thus, after converting the output of $D\exp_A$ to AIRM-orthonormal
coordinates at $Y$, the singular values of
\[
    (D\Exp_I)_A:
    (T_I\SPD(n),g_I)
    \longrightarrow
    (T_Y\SPD(n),g_Y)
\]
are $1$ on commuting directions and
\[
    \frac{
        \sinh\!\left(|a_i-a_j|/2\right)
    }{
        |a_i-a_j|/2
    }
\]
on the $(i,j)$ root directions.
Since $\SPD(n)$ with the affine-invariant metric is a Hadamard manifold,
$\Exp_I$ is globally invertible and
\[
    (D\Log_I)_Y
    =
    \bigl((D\Exp_I)_A\bigr)^{-1}.
\]
Its singular values are therefore the reciprocals:
$1$ and 
$
    \phi\!\left(\frac{|a_i-a_j|}{2}\right).
$
Because $\phi$ is decreasing on $[0,\infty)$, the smallest singular value is
attained at the largest spectral spread. Moreover,
\[
    \max_{i,j}|a_i-a_j|
    =
    \log
    \frac{
        \lambda_{\max}(X^{-1/2}YX^{-1/2})
    }{
        \lambda_{\min}(X^{-1/2}YX^{-1/2})
    }
    =
    \log q.
\]
Finally,
\[
    \phi\!\left(\frac{\log q}{2}\right)
    =
    \frac{\sqrt q\,\log q}{q-1},
\]
which proves \eqref{eq:spd-sigma-min} and
\eqref{eq:spd-condition-number-factor}.
\end{proof}

Theorem~\ref{thm:spd-exact-distortion} shows that geodesic distance alone
does not characterize logarithmic distortion on $\SPD(n)$. For example, if
$
    Y=cX$, 
$c>0,
$
then all logarithmic eigenvalues $a_i$ are equal and
\[
    \sigma_{\min}\!\bigl((D\Log_X)_Y\bigr)=1,
\]
although
\[
    d(X,Y)
    =
    \sqrt n\,|\log c|
\]
can be arbitrarily large. The relevant quantity for the distortion is thus
relative spectral spread rather than common logarithmic scale.

\begin{definition}[Spectral diameter]
\label{def:spectral-diameter}
For $\C\subset\SPD(n)$, define
\begin{equation}
\label{eq:spectral-diameter}
    \Delta_{\mathrm{spec}}(\C)
    :=
    \sup_{X,Y\in\C}
    \log\cond(X^{-1/2}YX^{-1/2}).
\end{equation}
\end{definition}

\begin{corollary}[SPD spectral scaling certificate]
\label{cor:spd-scaling-certificate}
Suppose that
\[
    \Delta_{\mathrm{spec}}(\C)<\infty.
\]
Then, 
\begin{equation}
\label{eq:spd-lambda-certificate}
    \Lambda_{\C}
    \geq
    \phi\!\left(
        \frac{\Delta_{\mathrm{spec}}(\C)}{2}
    \right).
\end{equation}
If, in addition, $\C\in\DGUC_p(\alpha)$, then every LMO configuration
satisfies
\begin{equation}
\label{eq:spd-prsi}
    \ip{w}{\Log_XV}_X
    \geq
    \alpha
    \phi\!\left(
        \frac{\Delta_{\mathrm{spec}}(\C)}{2}
    \right)
    D^{1-p}
    \norm{w}_X
    d(X,V)^p.
\end{equation}
\end{corollary}

\begin{proof}
Theorem~\ref{thm:spd-exact-distortion} and the monotonicity of $\phi$ give,
for every $X,Y\in\C$,
\[
    \sigma_{\min}\!\bigl((D\Log_X)_Y\bigr)
    \geq
    \phi\!\left(
        \frac{\Delta_{\mathrm{spec}}(\C)}{2}
    \right).
\]
Hence, \eqref{eq:spd-lambda-certificate} follows. The scaling estimate
\eqref{eq:spd-prsi} is then an immediate consequence of
Corollary~\ref{cor:singular-value-scaling}.
\end{proof}

\begin{remark}[Directional versus worst-case spectral distortion]
\label{rem:spd-directional}
The spectral diameter still gives a worst-case certificate. The exact
quantity in Theorem~\ref{thm:exact-jacobi-transfer} is directional. In the
root-space decomposition associated with
\eqref{eq:spd-relative-spectrum}, commuting components are undistorted,
whereas an $(i,j)$ component is weighted by
\[
    \phi\!\left(\frac{|a_i-a_j|}{2}\right).
\]
Consequently, $\Lambda_{\C}^{\LMO}$ can be substantially larger than the
global certificate in \eqref{eq:spd-lambda-certificate} when oracle
directions have little mass in the most anisotropic root spaces.
\end{remark}

\subsection{Consequences for the scaling exponent}
\label{subsec:certificate-exponent}

Combining this section with
Theorem~\ref{thm:exact-jacobi-transfer} separates the scaling margin into
the feasible-set modulus $\Delta_{\C}(r)$ and the ambient distortion factor
$\norm{(D\Log_x)_v^*w}/\norm{w}$. For a power-type modulus, the exponent
$p$ is supplied entirely by the feasible-set geometry. Ambient geometry enters through a multiplicative
conditioning factor: a comparison factor on a general manifold and an exact
root-space factor on $\SPD(n)$. Therefore, on any fixed regular region on
which this factor is uniformly bounded away from zero, ambient curvature
changes the scaling constant but not the power of $d(x,v)$.
Section~\ref{sec:upper-complexity} translates this separation into the RFW
complexity bound.

\section{Geometry-derived Riemannian Frank--Wolfe complexity}
\label{sec:upper-complexity}

The preceding sections separate two geometric ingredients that enter
projection-free convergence analysis. The power $p$ is inherited from the
feasible-set modulus, whereas the ambient manifold enters through the
logarithmic distortion factor. We now show how this separation propagates
to the Riemannian Frank--Wolfe complexity bound.

The analysis is stated for the standard gap-feedback short step. This choice
keeps the argument self-contained and makes explicit how the geometric
scaling certificates of Sections~\ref{sec:power-geometry}--
\ref{sec:certificates} enter the rate. More general RFW convergence results
under local error bounds and scaling assumptions, including open-loop
step-size rules, are developed in \cite{chen2026openloop}. Our purpose here
is different: the scaling inequality will be supplied by the geometry of the
feasible set and the differential of the logarithm map.

Throughout this section, let
\[
    h_t:=f(x_t)-f^\star,
    \qquad
    g_t:=
    -\ip{\nabla f(x_t)}{\Log_{x_t}v_t}_{x_t},
    \qquad
    r_t:=d(x_t,v_t),
\]
where
\[
    v_t\in
    \operatorname{LMO}_{\C}
    \bigl(x_t,-\nabla f(x_t)\bigr)
\]
is an exact Riemannian linear-minimization oracle. For $L>0$, we use the
short step
\begin{equation}
\label{eq:short-step}
    \eta_t
    :=
    \min\left\{
        \frac{g_t}{Lr_t^2},
        1
    \right\},
\end{equation}
whenever $r_t>0$. If $r_t=0$, then $v_t=x_t$ and hence $g_t=0$.
Under geodesic convexity, \eqref{eq:gap-dominates-subopt} then implies
$h_t=0$, so the algorithm may terminate.

\subsection{From an error bound to gradient growth}
\label{subsec:qeb-gradient}

The quadratic error bound supplies the objective-side estimate needed to
combine power-type scaling with the descent inequality.

\begin{lemma}[Error bound implies gradient growth]
\label{lem:qeb-gradient}
Suppose that $f$ is differentiable and geodesically convex on the relevant
feasible geodesics, and let $x\in\C$ satisfy
\[
    d(x,X^\star)
    \leq
    \mu\,h(x)^{1/2}.
\]
If $h(x)>0$, then
\begin{equation}
\label{eq:gradient-lower-qeb}
    \norm{\nabla f(x)}_x
    \geq
    \frac{h(x)^{1/2}}{\mu}.
\end{equation}
\end{lemma}

\begin{proof}
For every $x^\star\in X^\star$ for which the relevant minimizing geodesic
from $x$ to $x^\star$ is available, geodesic convexity gives
\[
    f(x^\star)
    \geq
    f(x)
    +
    \ip{\nabla f(x)}{\Log_xx^\star}_x.
\]
Therefore,
\[
    h(x)
    \leq
    -\ip{\nabla f(x)}{\Log_xx^\star}_x
    \leq
    \norm{\nabla f(x)}_x\,d(x,x^\star).
\]
Taking the infimum over $x^\star\in X^\star$, or an approximating sequence
if the distance is not attained, yields
\[
    h(x)
    \leq
    \norm{\nabla f(x)}_x\,d(x,X^\star).
\]
Using the quadratic error bound,
\[
    h(x)
    \leq
    \mu
    \norm{\nabla f(x)}_x
    h(x)^{1/2}.
\]
Since $h(x)>0$, division by $h(x)^{1/2}$ proves
\eqref{eq:gradient-lower-qeb}.
\end{proof}

Thus, strong geodesic convexity is not required as a primitive assumption.
For instance, whenever $m$-strong geodesic convexity and the corresponding
first-order optimality conditions imply
\[
    h(x)
    \geq
    \frac{m}{2}d(x,X^\star)^2,
\]
the quadratic error bound \eqref{eq:qeb} holds with
$
  \displaystyle  \mu=\sqrt{\frac{2}{m}}.
$

\subsection{Geometry-controlled oracle radius}
\label{subsec:oracle-radius-control}

We first derive the rate from an abstract power-scaling certificate. The
geometric results of the preceding sections will then provide the scaling
constant.

\begin{lemma}[Oracle-radius control]
\label{lem:oracle-radius-control}
Let $p\geq2$ and suppose that, at a nonoptimal iterate $x_t$, the oracle
triple satisfies
\begin{equation}
\label{eq:scaling-upper-section}
    g_t
    \geq
    \beta
    \norm{\nabla f(x_t)}_{x_t}
    D^{1-p}r_t^p
\end{equation}
for some $\beta>0$. If the quadratic error bound \eqref{eq:qeb} holds at
$x_t$, then
\begin{equation}
\label{eq:radius-vs-gap-error}
    r_t^2
    \leq
    D^{2(p-1)/p}
    \left(\frac{\mu}{\beta}\right)^{2/p}
    g_t^{2/p}h_t^{-1/p}.
\end{equation}
\end{lemma}

\begin{proof}
Since $h_t>0$, Lemma~\ref{lem:qeb-gradient} gives
\[
    \norm{\nabla f(x_t)}_{x_t}
    \geq
    \frac{h_t^{1/2}}{\mu}.
\]
Rearranging \eqref{eq:scaling-upper-section},
\[
    r_t^p
    \leq
    \frac{
        D^{p-1}g_t
    }{
        \beta\norm{\nabla f(x_t)}_{x_t}
    }
    \leq
    \frac{\mu D^{p-1}}{\beta}
    g_t h_t^{-1/2}.
\]
Raising both sides to the power $2/p$ yields
\eqref{eq:radius-vs-gap-error}.
\end{proof}

The inequality~\eqref{eq:radius-vs-gap-error} makes the separation explicit:
the power $p$ determines how the oracle radius depends on the gap, whereas
the geometric conditioning enters through the scalar $\beta$.

\subsection{A self-contained short-step rate}
\label{subsec:short-step-rate}

We next record the algorithmic consequence of power scaling and the
quadratic error bound.

\begin{theorem}[RFW under power scaling and a quadratic error bound]
\label{thm:abstract-upper-rate}
Let $p\geq2$ and $L>0$. Suppose that $f$ is geodesically convex and
$L$-smooth along the feasible minimizing geodesics used by the algorithm,
that an exact LMO exists at every nonoptimal iterate, and that the iterates
satisfy the quadratic error bound \eqref{eq:qeb} with constant $\mu>0$.
Assume further that every nonoptimal oracle triple along the trajectory
satisfies
\begin{equation}
\label{eq:abstract-scaling-rate}
    g_t
    \geq
    \beta
    \norm{\nabla f(x_t)}_{x_t}
    D^{1-p}r_t^p
\end{equation}
for a uniform constant $\beta>0$.
Then, the RFW iterates generated by \eqref{eq:rfw-update} with the short step
\eqref{eq:short-step} satisfy
\begin{equation}
\label{eq:two-branch-recursion}
    h_{t+1}
    \leq
    h_t
    -
    \min\left\{
        \frac{h_t}{2},
        c_p h_t^{\,2-1/p}
    \right\},
\end{equation}
where
\begin{equation}
\label{eq:cp-upper}
    c_p
    :=
    \frac{
        \beta^{2/p}
    }{
        2L\,D^{2(p-1)/p}\mu^{2/p}
    }.
\end{equation}

Consequently, either the method reaches an optimum in finitely many
iterations, or there exists a finite index $t_0$ such that, for every
$t\geq t_0$,
\begin{equation}
\label{eq:explicit-upper-rate}
    h_t
    \leq
    \left[
        h_{t_0}^{-(p-1)/p}
        +
        \frac{p-1}{p}
        c_p(t-t_0)
    \right]^{-p/(p-1)}.
\end{equation}
In particular,
\begin{equation}
\label{eq:upper-rate-p}
    h_t
    =
    O\!\left(t^{-p/(p-1)}\right).
\end{equation}
\end{theorem}

\begin{proof}
If $h_t=0$, there is nothing to prove. Assume therefore that $h_t>0$.
The smoothness inequality \eqref{eq:master-descent-prelim} gives
\begin{equation}
\label{eq:descent-upper-proof}
    h_{t+1}
    \leq
    h_t
    -
    \eta_tg_t
    +
    \frac{L}{2}\eta_t^2r_t^2.
\end{equation}
Consider first the full-step branch
\[
    g_t\geq Lr_t^2.
\]
Then, $\eta_t=1$, and
\[
    h_{t+1}
    \leq
    h_t-g_t+\frac{L}{2}r_t^2
    \leq
    h_t-\frac{g_t}{2}.
\]
Since geodesic convexity gives $g_t\geq h_t$,
\begin{equation}
\label{eq:full-step-half}
    h_{t+1}
    \leq
    \frac{h_t}{2}.
\end{equation}
On the interior branch, 
\[
    g_t<Lr_t^2,
\]
we have
\[
    \eta_t=\frac{g_t}{Lr_t^2},
\]
and \eqref{eq:descent-upper-proof} becomes
\begin{equation}
\label{eq:interior-decrease}
    h_{t+1}
    \leq
    h_t-\frac{g_t^2}{2Lr_t^2}.
\end{equation}
Lemma~\ref{lem:oracle-radius-control} implies
\[
    \frac{g_t^2}{2Lr_t^2}
    \geq
    c_p
    g_t^{\,2-2/p}
    h_t^{1/p}.
\]
Since $g_t\geq h_t$ and $2-2/p\geq0$,
\begin{equation}
\label{eq:power-step-decrease}
    \frac{g_t^2}{2Lr_t^2}
    \geq
    c_p h_t^{\,2-1/p}.
\end{equation}
Combining \eqref{eq:full-step-half} and
\eqref{eq:power-step-decrease} proves
\eqref{eq:two-branch-recursion}.
Set
\[
    a:=\frac{p-1}{p}\in(0,1),
\]
so that
\[
    2-\frac1p=1+a.
\]
Then, \eqref{eq:two-branch-recursion} implies
\begin{equation}
\label{eq:two-branch-a}
    h_{t+1}
    \leq
    h_t-
    \min\left\{
        \frac{h_t}{2},
        c_ph_t^{1+a}
    \right\}.
\end{equation}
Consider the threshold
\[
    \bar h
    :=
    \left(\frac{1}{2c_p}\right)^{1/a}.
\]
Whenever $h_t>\bar h$, one has
$c_ph_t^a>1/2$, and therefore
\eqref{eq:two-branch-a} gives
\[
    h_{t+1}\leq\frac{h_t}{2}.
\]
Hence, unless finite termination occurs earlier, the sequence enters the
region
\[
    h_t\leq\bar h
\]
after finitely many iterations. Let $t_0$ be an index from which this holds.
Since \eqref{eq:two-branch-a} also implies monotonicity of $h_t$, the same
inequality remains valid for all $t\geq t_0$.
For such $t$,
\[
    c_ph_t^a\leq\frac12,
\]
so the power term is the active term in
\eqref{eq:two-branch-a} and
\[
    h_{t+1}
    \leq
    h_t\bigl(1-c_ph_t^a\bigr).
\]
If $h_{t+1}=0$, finite termination has occurred. Otherwise, using
\[
    (1-z)^{-a}\geq1+az,
    \qquad 0\leq z<1,
\]
with $z=c_ph_t^a$, we obtain
\[
    h_{t+1}^{-a}
    \geq
    h_t^{-a}
    +
    ac_p.
\]
Summing from $t_0$ to $t-1$ gives
\[
    h_t^{-a}
    \geq
    h_{t_0}^{-a}
    +
    ac_p(t-t_0).
\]
Since $a=(p-1)/p$, inversion yields
\eqref{eq:explicit-upper-rate}, and hence
\eqref{eq:upper-rate-p}.
\end{proof}

\begin{remark}[Local error bounds]
\label{rem:local-qeb-rate}
Theorem~\ref{thm:abstract-upper-rate} is formulated in terms of conditions
holding along the trajectory so that the local accelerated regime is not
conflated with an entry argument. If the quadratic error bound and the
uniform scaling certificate hold only on a neighborhood $U$ of $X^\star$,
the same rate applies from any index $t_{\mathrm{ent}}$ after which the
iterates remain in $U$. Establishing such entry and invariance is a separate
question and is not assumed implicitly in the theorem.
\end{remark}

\subsection{Rate certified by Riemannian feasible-set geometry}
\label{subsec:geometry-derived-rate}

We now replace the abstract scaling assumption by the geometric certificate
obtained from the exact Jacobi transfer.

\begin{theorem}[Geometry-derived RFW complexity]
\label{thm:geometry-upper-rate}
Let $p\geq2$ and let
\[
    \C\in\DGUC_p(\alpha)
\]
have finite positive diameter $D$. Suppose that every nonoptimal oracle
configuration
\[
    (x_t,-\nabla f(x_t),v_t)
\]
lies in the regularity regime of
Theorem~\ref{thm:exact-jacobi-transfer} and satisfies the uniform
directional distortion bound
\begin{equation}
\label{eq:lambda-positive-upper}
    \frac{
        \norm{
            (D\Log_{x_t})_{v_t}^*
            \nabla f(x_t)
        }_{v_t}
    }{
        \norm{\nabla f(x_t)}_{x_t}
    }
    \geq
    \Lambda
    >0.
\end{equation}
Let $f$ be geodesically convex and $L$-smooth, with $L>0$, along the
feasible minimizing geodesics used by RFW, and suppose that the quadratic
error bound \eqref{eq:qeb} holds along the trajectory with constant
$\mu>0$.
Then, RFW with the short step \eqref{eq:short-step} satisfies
\begin{equation}
\label{eq:geometry-upper-rate}
    f(x_t)-f^\star
    =
    O\!\left(t^{-p/(p-1)}\right).
\end{equation}
More precisely, Theorem~\ref{thm:abstract-upper-rate} applies with
\begin{equation}
\label{eq:beta-geometry-upper}
    \beta=\alpha\Lambda,
\end{equation}
and hence
\begin{equation}
\label{eq:geometry-cp}
    c_p
    =
    \frac{
        (\alpha\Lambda)^{2/p}
    }{
        2L\,D^{2(p-1)/p}\mu^{2/p}
    }.
\end{equation}
\end{theorem}

\begin{proof}
Apply Theorem~\ref{thm:exact-jacobi-transfer} with
\[
    w=-\nabla f(x_t).
\]
For every nonoptimal iterate,
\[
    g_t
    \geq
    \alpha
    D^{1-p}r_t^p
    \norm{
        (D\Log_{x_t})_{v_t}^*
        \nabla f(x_t)
    }_{v_t}.
\]
Using \eqref{eq:lambda-positive-upper},
\[
    g_t
    \geq
    \alpha\Lambda
    \norm{\nabla f(x_t)}_{x_t}
    D^{1-p}r_t^p.
\]
Thus, the abstract scaling condition holds with
$\beta=\alpha\Lambda$, and the conclusion follows from
Theorem~\ref{thm:abstract-upper-rate}.
\end{proof}

\begin{remark}[Relation with the LMO-restricted distortion]
\label{rem:trajectory-lambda}
A sufficient condition for \eqref{eq:lambda-positive-upper} is
\[
    \Lambda_{\C}^{\LMO}\geq\Lambda>0.
\]
The trajectory-wise formulation in
\eqref{eq:lambda-positive-upper} is slightly sharper: the complexity theorem
requires a lower distortion bound only for oracle configurations actually
encountered by the algorithm, rather than uniformly over all possible LMO
configurations in $\C$.
\end{remark}

Theorem~\ref{thm:geometry-upper-rate} completes the passage from
feasible-set geometry to algorithmic complexity. The Frank--Wolfe scaling
margin is determined by the intrinsic feasible-set parameter $\alpha$ and
the logarithmic distortion encountered along the oracle trajectory. Thus
the convergence rate follows directly from the geometric quantities
identified by the Jacobi transfer principle.

\begin{corollary}[Curvature-certified rate]
\label{cor:rauch-upper-rate}
Under the objective-side and feasible-set assumptions of
Theorem~\ref{thm:geometry-upper-rate}, suppose that the relevant oracle
geodesics lie in a regular normal region satisfying
\[
    \Sec\geq-\kappa^2,
    \qquad
    \kappa\geq0.
\]
Then, Theorem~\ref{thm:abstract-upper-rate} applies with
\begin{equation}
\label{eq:rauch-beta-upper}
    \beta
    =
    \alpha
    \frac{\kappa D}{\sinh(\kappa D)},
\end{equation}
with the continuous interpretation $\beta=\alpha$ when $\kappa=0$.
Consequently,
\[
    f(x_t)-f^\star
    =
    O\!\left(t^{-p/(p-1)}\right).
\]
\end{corollary}

\begin{proof}
Since $x_t,v_t\in\C$,
\[
    d(x_t,v_t)\leq D.
\]
Corollary~\ref{cor:negative-curvature-certificate} gives
\[
    \sigma_{\min}
    \!\left(
        (D\Log_{x_t})_{v_t}
    \right)
    \geq
    \frac{\kappa D}{\sinh(\kappa D)}.
\]
Combining this estimate with the exact transfer theorem gives the abstract
scaling inequality with \eqref{eq:rauch-beta-upper}. The result follows from
Theorem~\ref{thm:abstract-upper-rate}.
\end{proof}

\begin{corollary}[SPD spectral rate]
\label{cor:spd-upper-rate}
Let
\[
    \C\subset\SPD(n),
    \qquad
    \C\in\DGUC_p(\alpha),
\]
and suppose that
\[
    \Delta_{\mathrm{spec}}(\C)<\infty.
\]
Under the objective-side assumptions of
Theorem~\ref{thm:geometry-upper-rate}, RFW with the short step
\eqref{eq:short-step} satisfies
\[
    f(x_t)-f^\star
    =
    O\!\left(t^{-p/(p-1)}\right),
\]
with the abstract scaling constant
\begin{equation}
\label{eq:spd-beta-upper}
    \beta
    =
    \alpha
    \phi\!\left(
        \frac{\Delta_{\mathrm{spec}}(\C)}{2}
    \right).
\end{equation}
\end{corollary}

\begin{proof}
Corollary~\ref{cor:spd-scaling-certificate} gives
\[
    \sigma_{\min}\!\bigl((D\Log_X)_V\bigr)
    \geq
    \phi\!\left(
        \frac{\Delta_{\mathrm{spec}}(\C)}{2}
    \right)
\]
for every relevant pair $X,V\in\C$. The exact transfer theorem therefore
gives the abstract power-scaling inequality with
\eqref{eq:spd-beta-upper}. Apply
Theorem~\ref{thm:abstract-upper-rate}.
\end{proof}

\subsection{Interpretation and relation to existing fast-rate analyses}
\label{subsec:upper-interpretation}

The rate
$
    O\!\left(t^{-p/(p-1)}\right)
$
inherits the separation established by the geometric transfer theorem.
The feasible-set modulus determines the power $p$ and hence the complexity
exponent $p/(p-1)$, while the ambient manifold contributes through the
logarithmic distortion. In Theorem~\ref{thm:abstract-upper-rate} this
distortion is represented by $\Lambda$; under sectional-curvature control
it is quantified by a Rauch comparison factor, and on $\SPD(n)$ by an
exact spectral factor. Consequently, whenever the relevant logarithmic
distortion remains uniformly positive along the oracle trajectory, the
ambient geometry determines the conditioning of the rate while the
feasible-set power determines its exponent.

Theorem~\ref{thm:abstract-upper-rate} provides the scaling-to-complexity
component of this framework. It converts a power-type Frank--Wolfe scaling
margin into the corresponding short-step complexity exponent, while the
preceding geometric results supply that margin from the feasible-set
modulus and the logarithmic differential. This complements analyses that
begin directly from error-bound and scaling assumptions, including the
local H\"olderian error bounds, directional scaling conditions, open-loop
rules, and gap-feedback policies studied by
Chen~\cite{chen2026openloop}. Here, the scaling law is generated
geometrically, so its power and conditioning retain an explicit
interpretation in terms of feasible-set geometry and ambient differential
distortion.

For $p=2$, the complexity bound specializes to $O(t^{-2})$. As $p$
increases, the exponent $p/(p-1)$ decreases to $1$, continuously
interpolating between the quadratic-set regime and the classical
$O(t^{-1})$ Frank--Wolfe scale. This agrees with the Euclidean
power-uniformly-convex frontier
\cite{kerdreux2026local,halbey2026curvature}. The sharpness result of
Section~\ref{sec:lower-bound} further shows that, for every fixed $p>3$,
the exponent $p/(p-1)$ is attained on a genuinely negatively curved
hyperbolic family under the same gap-feedback short-step policy.

\section{Sharpness on a hyperbolic hard family}
\label{sec:lower-bound}

The upper bound of Section~\ref{sec:upper-complexity} is sharp for the
gap-feedback Riemannian Frank--Wolfe policy on a genuinely curved manifold.
For every fixed $p>3$, we construct a compact feasible set in a hyperbolic
plane of fixed negative curvature for which the same short-step rule as in
Theorem~\ref{thm:geometry-upper-rate} generates trajectories satisfying
\[
    f(x_t)-f^\star
    =
    \Theta\!\left(t^{-p/(p-1)}\right).
\]
Thus, the geometry-derived exponent $p/(p-1)$ persists under nonzero ambient
curvature and is attained by the same policy for which the upper bound was
derived.

The construction combines three geometric and dynamical ingredients. A
Euclidean power-$p$ contact set is first embedded in a sufficiently small
Klein chart, where it retains a uniform double-geodesic power modulus.
After a weighted blow-up, the resulting Frank--Wolfe dynamics form a
$C^1$ perturbation of the corresponding flat slow dynamics. Finally, the
squared-distance objective centered at the boundary optimizer converts the
longitudinal slow variable into a quadratic objective gap. For $p>3$, the
strict contraction of the flat rescaled dynamics persists under the
hyperbolic perturbation, yielding the matching polynomial exponent.

Exact line search serves as an auxiliary reference dynamics for identifying
the limiting coefficients of the flat model. The sharpness result itself
is established for the gap-feedback short-step rule analyzed in
Section~\ref{sec:upper-complexity}. Consequently, the construction
establishes sharpness of the exponent for this policy over the
geometry-controlled class considered here.

\subsection{A genuinely hyperbolic power-contact family}
\label{subsec:hyperbolic-family}

Fix $\kappa>0$ and let $\mathbb H^2_{-\kappa^2}$ denote the complete simply
connected surface of constant sectional curvature $-\kappa^2$. We use the
Klein model
\[
    \mathbb K_\kappa^2
    :=
    \left\{
        z\in\mathbb R^2:
        \kappa\|z\|_2<1
    \right\}.
\]
Hyperbolic geodesics are represented by Euclidean chords in these
coordinates. Hence, every Euclidean convex set compactly contained in
$\mathbb K_\kappa^2$ is geodesically convex.
For sufficiently small $\rho>0$, set
\[
    c:=(-\rho,0)
\]
and define
\begin{equation}
\label{eq:klein-p-domain}
    \C_{p,\rho}
    :=
    c
    +
    \rho
    \left\{
        (z_1,z_2)\in\mathbb R^2:
        |z_1|^p+|z_2|^p\leq1
    \right\}
    \Subset
    \mathbb K_\kappa^2.
\end{equation}
Its rightmost boundary point is
\[
    x^\star
    =
    c+\rho(1,0)
    =
    0.
\]
This placement is useful for two reasons. The feasible boundary remains an
exact translated Euclidean $\ell_p$ boundary in Klein coordinates, while
the optimizer lies at the center of the model. At the origin,
\begin{equation}
\label{eq:normal-metric-expansion}
    g_{ij}(0)=\delta_{ij},
    \qquad
    \partial_k g_{ij}(0)=0,
    \qquad
    g_{ij}(z)
    =
    \delta_{ij}
    +
    O(\kappa^2\|z\|_2^2).
\end{equation}
Thus, the ambient metric has no linear jet at the contact point.

\subsubsection{An explicit Euclidean power margin}

We first isolate the Euclidean geometric input. This also makes the
membership of the hyperbolic hard family in $\DGUC_p$ independent of any
implicit uniform-convexity constant.
Let
\[
    B_p
    :=
    \left\{
        z\in\mathbb R^2:
        \|z\|_p\leq1
    \right\},
    \qquad p\geq2,
\]
and let
\[
    D_0
    :=
    \diam_2(B_p)
    =
    2^{3/2-1/p}.
\]

\begin{lemma}[Explicit Euclidean power margin]
\label{lem:explicit-euclidean-power-margin}
For every $p\geq2$, define
\begin{equation}
\label{eq:explicit-euclidean-delta}
    \delta_p^{\mathrm E}
    :=
    \frac{2^{1/p-1/2}}{p}.
\end{equation}
Then, for every $a,b\in B_p$, every $t\in[0,1]$, and
\[
    m_t:=(1-t)a+tb,
\]
one has
\begin{equation}
\label{eq:euclidean-power-ball}
    B_2\!\left(
        m_t,\,
        \delta_p^{\mathrm E}
        t(1-t)
        D_0^{1-p}
        \|a-b\|_2^p
    \right)
    \subseteq
    B_p.
\end{equation}
\end{lemma}

\begin{proof}
Clarkson's inequality for $p\geq2$ gives
\[
    \left\|
        \frac{a+b}{2}
    \right\|_p^p
    +
    \left\|
        \frac{a-b}{2}
    \right\|_p^p
    \leq
    \frac{
        \|a\|_p^p+\|b\|_p^p
    }{2}
    \leq1.
\]
Writing
\[
    d_p:=\|a-b\|_p,
\]
we obtain
\[
    \left\|
        \frac{a+b}{2}
    \right\|_p
    \leq
    \left(
        1-\frac{d_p^p}{2^p}
    \right)^{1/p}.
\]
Since $s\mapsto s^{1/p}$ is concave on $[0,1]$,
\[
    (1-s)^{1/p}
    \leq
    1-\frac{s}{p},
    \qquad
    0\leq s\leq1,
\]
and therefore,
\begin{equation}
\label{eq:midpoint-margin-p}
    1-
    \left\|
        \frac{a+b}{2}
    \right\|_p
    \geq
    \frac{d_p^p}{p2^p}.
\end{equation}
We now pass from the midpoint to an arbitrary $t$. Suppose first that
$0\leq t\leq1/2$. Then,
\[
    m_t
    =
    (1-2t)a
    +
    2t\frac{a+b}{2}.
\]
By convexity of the norm and \eqref{eq:midpoint-margin-p},
\[
    \|m_t\|_p
    \leq
    (1-2t)
    +
    2t
    \left(
        1-\frac{d_p^p}{p2^p}
    \right),
\]
so
\[
    1-\|m_t\|_p
    \geq
    \frac{2t}{p2^p}d_p^p.
\]
Since $t(1-t)\leq t$,
\[
    1-\|m_t\|_p
    \geq
    \frac{
        t(1-t)
    }{
        p2^{p-1}
    }
    d_p^p.
\]
The case $1/2\leq t\leq1$ follows by interchanging $a$ and $b$. Hence,
for all $t\in[0,1]$,
\begin{equation}
\label{eq:arbitrary-t-p-margin}
    1-\|m_t\|_p
    \geq
    \frac{
        t(1-t)
    }{
        p2^{p-1}
    }
    \|a-b\|_p^p.
\end{equation}
In dimension two and for $p\geq2$,
\[
    \|z\|_p
    \geq
    2^{1/p-1/2}\|z\|_2.
\]
Consequently,
\[
    \|a-b\|_p^p
    \geq
    2^{1-p/2}\|a-b\|_2^p.
\]
Substituting into \eqref{eq:arbitrary-t-p-margin} gives
\begin{equation}
\label{eq:euclidean-margin-2}
    1-\|m_t\|_p
    \geq
    \frac{
        2^{2-3p/2}
    }{p}
    t(1-t)
    \|a-b\|_2^p.
\end{equation}
Now, let $\zeta\in\mathbb R^2$ satisfy
\[
    \|\zeta\|_2
    \leq
    \frac{
        2^{2-3p/2}
    }{p}
    t(1-t)
    \|a-b\|_2^p.
\]
Since $\|\zeta\|_p\leq\|\zeta\|_2$ for $p\geq2$,
\[
    \|m_t+\zeta\|_p
    \leq
    \|m_t\|_p+\|\zeta\|_p
    \leq1.
\]
Thus, $m_t+\zeta\in B_p$.
Finally,
\[
    \delta_p^{\mathrm E}D_0^{1-p}
    =
    \frac{
        2^{2-3p/2}
    }{p},
\]
because
\[
    D_0=2^{3/2-1/p}.
\]
This proves \eqref{eq:euclidean-power-ball}.
\end{proof}

\begin{remark}[Relation with Euclidean uniform convexity]
\label{rem:euclidean-uniform-convexity-literature}
The power-$p$ uniform convexity of $\ell_p$ balls is classical and is part
of the Euclidean uniformly convex feasible-set framework used in
projection-free optimization. Lemma~\ref{lem:explicit-euclidean-power-margin}
is included to make the particular ball-inclusion constant required below
fully explicit; no novelty is claimed for the Euclidean statement.
\end{remark}
The explicit inclusion also gives the buffer needed to transport the
property through a small metric perturbation.

\begin{lemma}[Buffered Euclidean inclusion]
\label{lem:euclidean-buffered-margin}
Let
\[
    R_p(a,b,t)
    :=
    \delta_p^{\mathrm E}
    t(1-t)
    D_0^{1-p}
    \|a-b\|_2^p.
\]
For every $\theta\in[0,1]$ and every $\zeta$ satisfying
\[
    \|\zeta\|_2
    \leq
    (1-\theta)R_p(a,b,t),
\]
one has
\begin{equation}
\label{eq:euclidean-buffer}
    B_2\!\left(
        m_t+\zeta,\,
        \theta R_p(a,b,t)
    \right)
    \subseteq
    B_p.
\end{equation}
Equivalently,
\[
    \dist_2
    \left(
        m_t+\zeta,
        \mathbb R^2\setminus B_p
    \right)
    \geq
    \theta R_p(a,b,t).
\]
\end{lemma}

\begin{proof}
If $\|e\|_2\leq\theta R_p(a,b,t)$, then
\[
    \|\zeta+e\|_2
    \leq
    R_p(a,b,t).
\]
Lemma~\ref{lem:explicit-euclidean-power-margin} therefore gives
\[
    m_t+\zeta+e\in B_p.
\]
Since this holds for every such $e$, the claimed ball inclusion follows.
\end{proof}

\subsubsection{Stability under the hyperbolic metric}

We next show that the Euclidean margin survives the small hyperbolic
perturbation generated by the rescaled Klein chart.

\begin{lemma}[Stability of the double-geodesic power margin]
\label{lem:dguc-c2-stability}
Let $g_\varepsilon$ be a family of Riemannian metrics on a fixed
neighborhood of $B_p$ such that
\begin{equation}
\label{eq:metric-c2-small}
    \|g_\varepsilon-I\|_{C^2}
    \leq
    C\varepsilon,
    \qquad
    0<\varepsilon\leq\varepsilon_0.
\end{equation}
Assume in addition that the image of the $g_\varepsilon$-geodesic joining
any two points $a,b\in B_p$ is the Euclidean segment $[a,b]$. Let
$\lambda_\varepsilon(t)$ denote its Euclidean affine parameter when the
geodesic is parameterized at constant Riemannian speed. Assume uniformly
in $a,b,t$ that
\begin{equation}
\label{eq:lambda-comparison-stability}
    \lambda_\varepsilon(t)
    \bigl(1-\lambda_\varepsilon(t)\bigr)
    \geq
    (1-C\varepsilon)t(1-t).
\end{equation}
Then, there exist constants $c_p,C_p>0$, independent of $\varepsilon$, such
that
\begin{equation}
\label{eq:stable-dguc}
    \Delta_{B_p}^{g_\varepsilon}(r)
    \geq
    c_p\delta_p^{\mathrm E}
    (1-C_p\varepsilon)
    D_\varepsilon^{1-p}r^p,
\end{equation}
where $D_\varepsilon$ is the $g_\varepsilon$-diameter of $B_p$.
\end{lemma}

\begin{proof}
Fix $a,b\in B_p$ and $t\in[0,1]$, set
$\lambda=\lambda_\varepsilon(t)$, and denote the constant-speed geodesic
point by
\[
    m_\lambda=(1-\lambda)a+\lambda b.
\] Let
\[
    r_\varepsilon
    :=
    d_{g_\varepsilon}(a,b).
\]
Uniform equivalence of the metrics implied by
\eqref{eq:metric-c2-small} gives
\begin{align}
    r_\varepsilon
    &=
    \|a-b\|_2
    \bigl(1+O(\varepsilon)\bigr),
    \label{eq:distance-equivalence-stability}
    \\
    D_\varepsilon
    &=
    D_0
    \bigl(1+O(\varepsilon)\bigr),
    \label{eq:diameter-equivalence-stability}
    \\
    \|z\|_2
    &=
    \|z\|_{g_\varepsilon,m_\lambda}
    \bigl(1+O(\varepsilon)\bigr),
    \label{eq:tangent-equivalence-stability}
\end{align}
uniformly in the relevant variables.
The Christoffel symbols satisfy
\[
    \|\Gamma(g_\varepsilon)\|_{C^1}
    =
    O(\varepsilon).
\]
The geodesic equation therefore implies the local exponential estimate
\begin{equation}
\label{eq:exp-stability}
    \Exp_{m_\lambda}^{g_\varepsilon}(z)
    =
    m_\lambda+z+E_\varepsilon(m_\lambda,z),
\end{equation}
where, uniformly for the perturbations considered below,
\begin{equation}
\label{eq:exp-stability-error}
    \|E_\varepsilon(m_\lambda,z)\|_2
    \leq
    C\varepsilon\|z\|_2.
\end{equation}
Choose $c_p>0$ sufficiently small and $C_p>0$ sufficiently large. Suppose
that
\begin{equation}
\label{eq:stable-riemannian-radius}
    \|z\|_{g_\varepsilon,m_\lambda}
    \leq
    c_p\delta_p^{\mathrm E}
    (1-C_p\varepsilon)
    t(1-t)
    D_\varepsilon^{1-p}
    r_\varepsilon^p.
\end{equation}
Using
\eqref{eq:distance-equivalence-stability}--
\eqref{eq:tangent-equivalence-stability}, and absorbing fixed
norm-equivalence constants into $c_p$, we obtain
\begin{equation}
\label{eq:stable-radius-euclidean}
    \|z\|_2
    \leq
    (1-\theta_\varepsilon)
    R_p(a,b,\lambda),
\end{equation}
where
\[
    \theta_\varepsilon
    \geq
    c\varepsilon
\]
for some constant $c>0$ independent of $a,b,t$.
Lemma~\ref{lem:euclidean-buffered-margin} therefore gives
\begin{equation}
\label{eq:stable-available-buffer}
    \dist_2
    \left(
        m_\lambda+z,
        \mathbb R^2\setminus B_p
    \right)
    \geq
    \theta_\varepsilon
    R_p(a,b,\lambda).
\end{equation}
On the other hand,
\eqref{eq:exp-stability-error} and
\eqref{eq:stable-radius-euclidean} imply
\[
    \|E_\varepsilon(m_\lambda,z)\|_2
    \leq
    C'\varepsilon
    R_p(a,b,\lambda).
\]
Taking $C_p$ sufficiently large ensures that
\[
    C'\varepsilon
    \leq
    \theta_\varepsilon.
\]
Hence, the exponential endpoint cannot leave $B_p$:
\[
    \Exp_{m_\lambda}^{g_\varepsilon}(z)\in B_p.
\]
The estimate is uniform as $t\to0,1$ and as $a\to b$, because both the
available Euclidean buffer and the exponential-map error vanish with the
same factor
\[
    t(1-t)\|a-b\|_2^p.
\]
This proves \eqref{eq:stable-dguc}.
\end{proof}

\begin{corollary}[Power geometry of the hyperbolic hard family]
\label{cor:hard-set-dguc}
Fix $p\geq2$ and $\kappa>0$. There exist constants $c_p,C_p>0$ and
$\rho_0=\rho_0(p,\kappa)>0$ such that, for every
$0<\rho\leq\rho_0$,
\begin{equation}
\label{eq:hard-set-alpha-quantitative}
    \C_{p,\rho}
    \in
    \DGUC_p(\alpha_{p,\kappa,\rho}),
    \qquad
    \alpha_{p,\kappa,\rho}
    \geq
    c_p\delta_p^{\mathrm E}
    \left(
        1-C_p\kappa^2\rho^2
    \right)
    >0.
\end{equation}
\end{corollary}

\begin{proof}
Pull back the hyperbolic metric under the affine chart map
\[
    \Phi_\rho(\bar z)
    :=
    c+\rho\bar z.
\]
After division by the constant metric scale $\rho^2$, the rescaled pullback
metric on a fixed neighborhood of $B_p$ satisfies, after a fixed linear
normalization,
\begin{equation}
\label{eq:klein-rescaled-c2}
    \left\|
        \rho^{-2}\Phi_\rho^\ast g_\kappa-I
    \right\|_{C^2}
    \leq
    C\kappa^2\rho^2.
\end{equation}
Moreover, Klein geodesics are Euclidean chords. Along a fixed chord, the
rescaled metric changes its speed density by a factor $1+O(\kappa^2\rho^2)$
uniformly on $B_p$. Writing the constant-speed parameter as
$\lambda_{\rho}(t)$ in Euclidean affine coordinates and comparing the
partial and total arc-length integrals gives
\[
    \lambda_{\rho}(t)
    \bigl(1-\lambda_{\rho}(t)\bigr)
    \geq
    \bigl(1-C\kappa^2\rho^2\bigr)t(1-t),
\]
uniformly in the endpoints and in $t\in[0,1]$. Thus the parameter comparison
required in Lemma~\ref{lem:dguc-c2-stability} holds as well.
Applying Lemma~\ref{lem:dguc-c2-stability} with
\[
    \varepsilon
    =
    \kappa^2\rho^2
\]
gives a positive power-$p$ double-geodesic modulus for the rescaled metric.
Finally, the normalization in Definition~\ref{def:dguc} is invariant under
constant rescaling of the metric, as established in
Proposition~\ref{prop:scale-invariance}. Therefore, the same dimensionless
parameter transfers back to $\C_{p,\rho}$.

Choosing $\rho_0$ so that
\[
    C_p\kappa^2\rho_0^2<\frac12
\]
ensures positivity of the right-hand side of
\eqref{eq:hard-set-alpha-quantitative}.
\end{proof}
Consider now the objective
\begin{equation}
\label{eq:hyperbolic-objective}
    f_\star(x)
    :=
    \frac12
    d_\kappa(x,x^\star)^2,
    \qquad
    x\in\C_{p,\rho}.
\end{equation}
The point $x^\star$ is both the boundary contact point and the unique
minimizer. In particular,
\[
    \nabla f_\star(x^\star)=0.
\]
This removes the linear normal term that would arise from an external-center
squared-distance objective.
We parameterize points in the slow region by
\begin{equation}
\label{eq:slow-coordinates}
    x(u,y)
    =
    x^\star
    +
    \rho
    \left(
        -u,
        \sigma y u^{1+\alpha_p}
    \right),
    \qquad
    \alpha_p:=\frac{p-1}{p},
    \qquad
    \sigma\in\{-1,1\}.
\end{equation}
The variable $u>0$ is the longitudinal depth, while $y>0$ is the normalized
transverse magnitude. The sign $\sigma$ alternates along the slow orbit.
Near the contact point, the boundary satisfies
\begin{equation}
\label{eq:p-contact-expansion}
    1-(1-|s|^p)^{1/p}
    =
    \frac1p|s|^p
    +
    O(|s|^{2p}),
\end{equation}
so the natural weighted blow-up is
\begin{equation}
\label{eq:epsilon-blowup}
    \varepsilon:=u^{1/p},
    \qquad
    s=\varepsilon z,
    \qquad
    u=\varepsilon^p.
\end{equation}

\begin{remark}[Why centering the objective preserves the oracle scale]
\label{rem:objective-repair-scale}
In the flat normal form, the displacement from $x^\star$ has components
\[
    \left(
        -u,
        \sigma y u^{2-1/p}
    \right).
\]
The gradient of \eqref{eq:hyperbolic-objective} has the same leading
scales. Their ratio is $u^{(p-1)/p}$, and the duality map of the
$\ell_p$ boundary therefore places the LMO endpoint at transverse scale
$u^{1/p}$. Centering the objective at $x^\star$ removes the first-order
objective gap without removing the singular oracle scale responsible for
the slow dynamics.
\end{remark}

\subsection{Uniform local expansions}
\label{subsec:hyperbolic-expansions}

We next record the normalized expansion used to transfer the flat slow
dynamics to fixed negative curvature.

\begin{lemma}[Normalized hyperbolic expansion]
\label{lem:uniform-hyperbolic-expansion}
Fix a compact interval
\[
    I\Subset(0,\infty)
\]
containing the flat slow fixed point $C_p$. For sufficiently small fixed
$\rho>0$, there exist $u_0>0$, $C_I<\infty$, smooth functions
$B_p,G_p:I\to\mathbb R$ with $G_p>0$, and remainders $R_v,R_\tau$ such that,
uniformly for $y\in I$ and $0<u\leq u_0$,
\begin{align}
    (v_\kappa-x)_1
    &=
    u
    \left[
        B_p(y)+R_v(u,y)
    \right],
    &
    \|R_v(u,\cdot)\|_{C^1(I)}
    &\leq
    C_I(1+\kappa^2\rho^2)u^{2/p},
    \label{eq:lmo-longitudinal-expansion}
    \\
    \tau_\kappa(u,y)
    &=
    u^{2(p-1)/p}
    \left[
        G_p(y)+R_\tau(u,y)
    \right],
    &
    \|R_\tau(u,\cdot)\|_{C^1(I)}
    &\leq
    C_I(1+\kappa^2\rho^2)u^{2/p}.
    \label{eq:line-search-expansion}
\end{align}
Moreover, after the local boundary jet is held fixed, the contribution of
the ambient hyperbolic metric to either normalized remainder is
$O(\kappa^2u^{2/p})$.
\end{lemma}

\begin{proof}
Set
\[
    \varepsilon=u^{1/p}
\]
and write the boundary parameter as $s=\varepsilon z$. Since
$x^\star=0$, the boundary has the exact representation
\begin{equation}
\label{eq:boundary-exact-2k}
    v(\varepsilon z)
    =
    \rho
    \left(
        (1-|\varepsilon z|^p)^{1/p}-1,
        \varepsilon z
    \right).
\end{equation}
Uniformly, for $z$ in a fixed compact set,
\begin{equation}
\label{eq:boundary-weighted-jet}
    v(\varepsilon z)
    =
    \rho
    \left(
        -\frac1p\varepsilon^p|z|^p,
        \varepsilon z
    \right)
    +
    \rho
    \left(
        O(\varepsilon^{2p}),
        0
    \right),
\end{equation}
with the same order after one $z$-derivative. The current point is
\begin{equation}
\label{eq:current-weighted-2k}
    x(\varepsilon,y)
    =
    \rho
    \left(
        -\varepsilon^p,
        \sigma y\varepsilon^{2p-1}
    \right).
\end{equation}
Hence, every point entering the local oracle calculation is
$O(\rho\varepsilon)$ from the origin.
At the center of the Klein model,
\begin{equation}
\label{eq:normal-even-jet}
    g_{ij}(\xi)
    =
    \delta_{ij}
    -
    \frac13
    R_{ikjl}(0)\xi^k\xi^l
    +
    O(\kappa^4\|\xi\|^4),
    \qquad
    \Gamma_{ij}^k(\xi)
    =
    O(\kappa^2\|\xi\|),
\end{equation}
with corresponding $C^1$ bounds. Thus, the coefficients entering the
first-order stationarity equations differ from their flat counterparts by
$O(\kappa^2\rho^2\varepsilon^2)$. In particular, no
$O(\varepsilon)$ metric correction occurs.

For the LMO, substitute
\eqref{eq:boundary-weighted-jet} and
\eqref{eq:current-weighted-2k} into the stationarity equation and divide by
its common leading power of $\varepsilon$. The normalized equation has the
form
\begin{equation}
\label{eq:normalized-lmo-map-2k}
    \mathscr F_\kappa(\varepsilon,y,z)
    =
    F_0(y,z)
    +
    \varepsilon^2
    F_{2,\kappa,\rho}(\varepsilon,y,z)
    +
    \varepsilon^p
    F_p^{\mathrm{bdry}}(\varepsilon,y,z),
\end{equation}
where the remainder functions and their first $y$-derivatives are uniformly
bounded on the relevant compact set.
Let $z_0(y)$ denote the flat root. Shrinking $I$ if necessary, assume
\[
    \left|
        \partial_zF_0(y,z_0(y))
    \right|
    \geq
    m_I>0.
\]
The uniform implicit-function theorem gives
\begin{equation}
\label{eq:lmo-root-2k}
    z_\kappa(\varepsilon,y)
    =
    z_0(y)
    +
    O(\varepsilon^2),
    \qquad
    \partial_yz_\kappa
    =
    \partial_yz_0
    +
    O(\varepsilon^2).
\end{equation}
Substitution into the longitudinal boundary component proves
\eqref{eq:lmo-longitudinal-expansion}.
For the auxiliary exact-line-search calculation, write
\[
    \tau=\varepsilon^{2p-2}q
\]
and let $X(\tau)$ be the geodesic from the current point toward the LMO
point. The interior stationarity equation is
\begin{equation}
\label{eq:els-stationarity-2k}
    \left\langle
        \Log_{X(\tau)}x^\star,
        \dot X(\tau)
    \right\rangle_{X(\tau)}
    =
    0.
\end{equation}
After the same normalization, one obtains
\begin{equation}
\label{eq:normalized-line-map-2k}
    \mathscr L_\kappa(\varepsilon,y,q)
    =
    L_0(y,q)
    +
    \varepsilon^2
    L_{2,\kappa,\rho}(\varepsilon,y,q)
    +
    \varepsilon^p
    L_p^{\mathrm{bdry}}(\varepsilon,y,q).
\end{equation}
The flat minimizing root
\[
    q_0(y)=G_p(y)
\]
is positive and simple on $I$. A second uniform implicit-function argument
therefore gives
\begin{equation}
\label{eq:line-root-2k}
    q_\kappa(\varepsilon,y)
    =
    q_0(y)+O(\varepsilon^2),
    \qquad
    \partial_yq_\kappa
    =
    \partial_yq_0+O(\varepsilon^2).
\end{equation}
Restoring $\tau=\varepsilon^{2p-2}q$ proves
\eqref{eq:line-search-expansion}. Since $p\geq2$, the boundary correction
$O(\varepsilon^p)$ is never of lower order than the
$O(\varepsilon^2)$ metric correction.
\end{proof}

\begin{remark}[Role of the centered Klein placement]
\label{rem:centered-klein-placement}
The absence of an $O(u^{1/p})$ normalized metric term is not a consequence
of smoothness alone. It uses the placement $x^\star=0$: the feasible
boundary retains its exact power-$p$ contact while the ambient metric has
no linear jet at the contact point. This is why the first normalized
curvature correction begins at order $u^{2/p}$.
\end{remark}

\subsection{Persistence under the same gap-feedback short step}
\label{subsec:short-step-hard}

We now analyze the policy used in
Theorem~\ref{thm:geometry-upper-rate}. Let $L_\rho$ denote a valid uniform
smoothness constant of
\[
    f_\star(x)
    =
    \frac12d_\kappa(x,x^\star)^2
\]
on $\C_{p,\rho}$. In the hyperbolic plane, Hessian comparison gives
\begin{equation}
\label{eq:hessian-hyperbolic-squareddistance}
    \lambda_{\max}
    \bigl(
        \operatorname{Hess}f_\star(x)
    \bigr)
    =
    \max\left\{
        1,
        \kappa r\coth(\kappa r)
    \right\},
    \qquad
    r=d_\kappa(x,x^\star).
\end{equation}
Since
\[
    \diam(\C_{p,\rho})=O(\rho),
\]
we may choose
\begin{equation}
\label{eq:Lrho-expansion}
    L_\rho
    =
    1+O(\kappa^2\rho^2).
\end{equation}
RFW is run with exactly the gap-feedback rule
\begin{equation}
\label{eq:hard-short-step}
    \eta_{\mathrm{ss}}(x)
    =
    \min\left\{
        \frac{
            g(x)
        }{
            L_\rho d_\kappa(x,v(x))^2
        },
        1
    \right\}.
\end{equation}
For reference, the flat normalized transverse dynamics are described by
\begin{equation}
\label{eq:flat-slow-map}
    F_p(y)
    :=
    y^{1-q}-\frac{y}{p},
    \qquad
    q:=\frac{p}{p-1}.
\end{equation}
Let $C_p>0$ denote its positive fixed point. Direct differentiation gives
\begin{equation}
\label{eq:slow-multiplier}
    F_p'(C_p)
    =
    -\frac{2}{p-1}.
\end{equation}
Hence,
\begin{equation}
\label{eq:spectral-gap-pgt3}
    |F_p'(C_p)|<1
    \qquad\Longleftrightarrow\qquad
    p>3.
\end{equation}

\begin{lemma}[Normalized short-step dynamics]
\label{lem:short-step-expansion}
Fix $p>3$ and a compact interval $I$ containing $C_p$ in its interior.
For sufficiently small fixed $\rho>0$, uniformly for $y\in I$ and
$u\downarrow0$,
\begin{equation}
\label{eq:short-step-scale-hard}
    \eta_{\mathrm{ss}}(u,y)
    =
    u^{2(p-1)/p}
    \left[
        \frac{G_p(y)}{L_\rho}
        +
        R_{\mathrm{ss}}(u,y)
    \right],
\end{equation}
where
\begin{equation}
\label{eq:short-step-remainder}
    \|R_{\mathrm{ss}}(u,\cdot)\|_{C^1(I)}
    \leq
    C u^{2/p}.
\end{equation}
In particular, the cap in \eqref{eq:hard-short-step} is inactive for all
sufficiently small $u$.
The normalized transverse update has the form
\begin{equation}
\label{eq:short-slow-map}
    y^+
    =
    F_{p,\rho}^{\mathrm{ss}}(y)
    +
    R_y^{\mathrm{ss}}(u,y),
\end{equation}
where
\begin{equation}
\label{eq:short-limiting-map}
    \|
        F_{p,\rho}^{\mathrm{ss}}-F_p
    \|_{C^1(I)}
    \leq
    C\kappa^2\rho^2,
\end{equation}
and
\begin{equation}
\label{eq:short-state-remainder}
    \|
        R_y^{\mathrm{ss}}(u,\cdot)
    \|_{C^1(I)}
    \leq
    C u^{2/p}.
\end{equation}
\end{lemma}

\begin{proof}
For the flat quadratic model
\[
    f_0(z)=\frac12\|z\|_2^2,
\]
the interior exact minimizer along the chord from $x$ to $v$ is
\[
    \tau_0
    =
    -\frac{
        \langle x,v-x\rangle
    }{
        \|v-x\|_2^2
    }
    =
    \frac{g_0}{\|v-x\|_2^2}.
\]
Its normalized coefficient is $G_p(y)$.
The centered hyperbolic expansions give, uniformly in $C^1(I)$,
\[
    g_\kappa(u,y)
    =
    g_0(u,y)
    \left[
        1+O(u^{2/p})
    \right]
\]
and
\[
    d_\kappa(x,v)^2
    =
    \|v-x\|_2^2
    \left[
        1+
        O(\kappa^2\rho^2u^{2/p})
    \right].
\]
Consequently,
\[
    \frac{
        g_\kappa(u,y)
    }{
        L_\rho d_\kappa(x,v)^2
    }
    =
    u^{2(p-1)/p}
    \left[
        \frac{G_p(y)}{L_\rho}
        +
        O(u^{2/p})
    \right],
\]
which proves
\eqref{eq:short-step-scale-hard}--\eqref{eq:short-step-remainder}.
Since $2(p-1)/p>0$, the ratio tends to zero and the cap is eventually
inactive.

Substitution into the longitudinal and transverse coordinate updates gives
\eqref{eq:short-slow-map}. The fixed-domain difference between the limiting
short-step map and the flat map is generated by
\[
    L_\rho^{-1}-1
    =
    O(\kappa^2\rho^2),
\]
whereas all state-dependent geometric corrections vanish as
$u\downarrow0$. This yields
\eqref{eq:short-limiting-map} and
\eqref{eq:short-state-remainder}.
\end{proof}

\begin{lemma}[Persistence of the short-step slow tube]
\label{lem:short-step-slow-tube}
For every fixed $p>3$ and $\kappa>0$, there exists $\rho_0>0$ such that,
for every $0<\rho\leq\rho_0$, the limiting map
$F_{p,\rho}^{\mathrm{ss}}$ has an attracting positive fixed point
\[
    C_{p,\rho}^{\mathrm{ss}}
    =
    C_p+O(\kappa^2\rho^2).
\]
Moreover, there exists a nonempty invariant slow tube for the full
short-step dynamics. For every trajectory initialized in a sufficiently
small invariant sub-tube,
\[
    y_t
    =
    C_{p,\rho}^{\mathrm{ss}}
    +
    o(1),
\]
and there exist constants
$0<a_-\leq a_+<\infty$ such that, for all sufficiently large $t$,
\begin{equation}
\label{eq:short-u-recurrence}
    a_-u_t^{1+\chi}
    \leq
    u_t-u_{t+1}
    \leq
    a_+u_t^{1+\chi},
    \qquad
    \chi:=\frac{2(p-1)}{p}.
\end{equation}
\end{lemma}

\begin{proof}
By \eqref{eq:slow-multiplier},
\[
    |F_p'(C_p)|
    =
    \frac{2}{p-1}
    <1.
\]
Thus, $C_p$ is a hyperbolic attracting fixed point of the flat map. By
\eqref{eq:short-limiting-map},
$F_{p,\rho}^{\mathrm{ss}}$ is a sufficiently small $C^1$ perturbation of
$F_p$ when $\rho$ is small. The implicit-function theorem and persistence
of hyperbolic fixed points therefore give an attracting fixed point
\[
    C_{p,\rho}^{\mathrm{ss}}
    =
    C_p+O(\kappa^2\rho^2)
\]
and a uniform contraction neighborhood.

The state-dependent perturbation
\eqref{eq:short-state-remainder} tends to zero as $u\downarrow0$ and
therefore preserves a smaller invariant tube. The longitudinal update
obtained from Lemma~\ref{lem:short-step-expansion} has leading order
\[
    u-u^+
    =
    u^{1+\chi}
    \left[
        A_{p,\rho}^{\mathrm{ss}}(y)
        +
        o(1)
    \right],
\]
where
\[
    A_{p,\rho}^{\mathrm{ss}}
    \bigl(
        C_{p,\rho}^{\mathrm{ss}}
    \bigr)
    >0
\]
for sufficiently small $\rho$, by continuity from the flat dynamics.
Shrinking the tube once more gives the two-sided recurrence
\eqref{eq:short-u-recurrence}.
\end{proof}

\begin{lemma}[Discrete power law]
\label{lem:discrete-power-law}
Suppose a positive sequence satisfies
\[
    a_-u_t^{1+\chi}
    \leq
    u_t-u_{t+1}
    \leq
    a_+u_t^{1+\chi}
\]
for all sufficiently large $t$, where $\chi>0$. Then
\[
    u_t
    =
    \Theta(t^{-1/\chi}).
\]
In particular, for
\[
    \chi=\frac{2(p-1)}{p},
\]
the short-step slow tube satisfies
\begin{equation}
\label{eq:u-rate-p}
    u_t
    =
    \Theta\!\left(
        t^{-p/[2(p-1)]}
    \right).
\end{equation}
\end{lemma}

\begin{proof}
The recurrence implies
\[
    \frac{u_t-u_{t+1}}{u_t}
    =
    O(u_t^\chi)
    \longrightarrow0,
\]
and hence,
\[
    \frac{u_{t+1}}{u_t}\longrightarrow1.
\]
Applying the mean-value theorem to
$s\mapsto s^{-\chi}$ shows that
\[
    u_{t+1}^{-\chi}
    -
    u_t^{-\chi}
\]
is bounded above and below by positive constants for all sufficiently large
$t$. Summation yields
\[
    u_t^{-\chi}
    =
    \Theta(t),
\]
and therefore,
\[
    u_t
    =
    \Theta(t^{-1/\chi}).
\]
Substituting $\chi=2(p-1)/p$ proves
\eqref{eq:u-rate-p}.
\end{proof}

\subsection{Quadratic gap conversion and same-policy sharpness}
\label{subsec:matching-lower-bound}

It remains to convert the slow state variable into objective error. Let
\[
    \xi(u,y)
    :=
    \Log_{x^\star}x(u,y).
\]
Because the metric is Euclidean to first order at $x^\star=0$, uniformly
for $y$ in the invariant tube,
\begin{equation}
\label{eq:xi-slow-expansion}
    \|\xi(u,y)\|^2
    =
    \rho^2u^2
    \bigl(
        1+o(1)
    \bigr).
\end{equation}
Indeed, the squared transverse contribution has order
\[
    u^{2(2-1/p)}
    =
    o(u^2)
\]
for $p>1$. Since
\[
    f_\star(x)
    =
    \frac12
    \|\Log_{x^\star}x\|_{x^\star}^2,
\]
we obtain
\begin{equation}
\label{eq:objective-quadratic-gap}
    f_\star(x(u,y))
    -
    f_\star(x^\star)
    =
    \frac{\rho^2}{2}
    u^2
    \bigl(
        1+o(1)
    \bigr).
\end{equation}
Thus, the objective gap is uniformly $\Theta(u^2)$ along the slow tube.

\begin{theorem}[Same-policy hyperbolic sharpness]
\label{thm:hyperbolic-short-step-lower}
For every fixed $p>3$, there exist $\kappa>0$, $\rho>0$, a compact
geodesically convex set
\[
    \C_{p,\rho}
    \subset
    \mathbb H^2_{-\kappa^2},
\]
and an initialization in its short-step slow tube such that RFW applied to
\[
    \min_{x\in\C_{p,\rho}}
    \frac12d_\kappa(x,x^\star)^2
\]
with the same gap-feedback rule \eqref{eq:short-step}, using the valid
smoothness constant $L_\rho$, satisfies
\begin{equation}
\label{eq:hyperbolic-short-step-rate}
    f_\star(x_t)-f_\star(x^\star)
    =
    \Theta\!\left(
        t^{-p/(p-1)}
    \right).
\end{equation}
Moreover, the constructed instance satisfies the geometric and
objective-side assumptions of
Theorem~\ref{thm:geometry-upper-rate}. Hence, for every fixed $p>3$, the
exponent $p/(p-1)$ of that theorem is sharp for the analyzed
gap-feedback short-step policy over the stated geometry-controlled class.
\end{theorem}

\begin{proof}
Lemma~\ref{lem:short-step-slow-tube} and
Lemma~\ref{lem:discrete-power-law} give
\[
    u_t
    =
    \Theta\!\left(
        t^{-p/[2(p-1)]}
    \right).
\]
Equation~\eqref{eq:objective-quadratic-gap} then yields
\[
    f_\star(x_t)-f_\star(x^\star)
    =
    \Theta(u_t^2)
    =
    \Theta\!\left(
        t^{-p/(p-1)}
    \right).
\]
It remains to verify membership in the class of
Theorem~\ref{thm:geometry-upper-rate}.
Corollary~\ref{cor:hard-set-dguc} gives
\[
    \C_{p,\rho}
    \in
    \DGUC_p(\alpha_{p,\kappa,\rho})
\]
with a positive parameter for sufficiently small $\rho$. Hyperbolic space
is Hadamard, so the logarithm is globally single valued and regular.
Moreover,
\[
    f_\star(x)
    =
    \frac12d_\kappa(x,x^\star)^2
\]
is smooth and geodesically strongly convex, and
\[
    d_\kappa(x,x^\star)
    =
    \sqrt{
        2\bigl(
            f_\star(x)-f_\star(x^\star)
        \bigr)
    },
\]
so the quadratic error bound holds with
\[
    \mu=\sqrt2.
\]
Finally, since $\C_{p,\rho}$ has finite diameter, the hyperbolic
logarithmic-distortion certificate gives a positive uniform lower bound on
the relevant differential factor. Thus all assumptions of
Theorem~\ref{thm:geometry-upper-rate} are satisfied.
\end{proof}
\begin{remark}[Scope of the sharpness statement]
\label{rem:sharpness-scope}
Theorem~\ref{thm:hyperbolic-short-step-lower} establishes same-policy
sharpness for the gap-feedback Riemannian Frank--Wolfe rule analyzed in
Theorem~\ref{thm:geometry-upper-rate}. For every fixed $p>3$, the
geometry-controlled problem class contains a genuinely curved instance on
which this policy attains the rate
\[
    \Theta\!\left(t^{-p/(p-1)}\right).
\]
Hence, the exponent $p/(p-1)$ in the upper bound is sharp for the
gap-feedback short-step policy over this class.
\end{remark}

\subsection{Exact line search as a reference dynamics}
\label{subsec:slow-dynamics}

The exact-line-search dynamics provides the reference model for the local
construction. The same analysis yields the corresponding rate and identifies
the flat coefficient $G_p$ governing the limiting dynamics of the
gap-feedback short-step scheme.
For exact line search, composing the oracle and line-search expansions gives
\begin{equation}
\label{eq:longitudinal-update}
    u-u^+
    =
    u^{1+\chi}
    \left[
        A_p(y)+R_u(u,y)
    \right],
    \qquad
    \chi=\frac{2(p-1)}{p},
\end{equation}
where $A_p$ is positive near $C_p$ and
\[
    |R_u(u,y)|
    \leq
    C_I(1+\kappa^2\rho^2)u^{2/p}.
\]
The normalized transverse map satisfies
\[
    y^+
    =
    F_p(y)
    +
    R_y(u,y),
\]
with
\[
    \|R_y(u,\cdot)\|_{C^1(I)}
    \longrightarrow0
    \qquad
    \text{as }u\downarrow0.
\]
Since
\[
    |F_p'(C_p)|<1
    \qquad
    (p>3),
\]
the same contraction argument yields an invariant slow tube and
\[
    u_t
    =
    \Theta\!\left(
        t^{-p/[2(p-1)]}
    \right).
\]

\begin{proposition}[Exact-line-search reference rate]
\label{prop:hyperbolic-exact-line-search}
For every fixed $p>3$, the hyperbolic construction above admits an
initialization for which exact-line-search RFW satisfies
\begin{equation}
\label{eq:hyperbolic-matching-rate}
    f_\star(x_t)-f_\star(x^\star)
    =
    \Theta\!\left(
        t^{-p/(p-1)}
    \right).
\end{equation}
\end{proposition}

\begin{proof}
The exact-line-search slow dynamics give
\[
    u_t
    =
    \Theta\!\left(
        t^{-p/[2(p-1)]}
    \right),
\]
while \eqref{eq:objective-quadratic-gap} gives
\[
    f_\star(x_t)-f_\star(x^\star)
    =
    \Theta(u_t^2).
\]
Combining the two estimates proves the claim.
\end{proof}

\begin{remark}[Why the external-center construction does not match]
\label{rem:external-center-failure}
If the squared-distance center is moved to a point
\[
    a=\Exp_{x^\star}(\ell n)
\]
outside the feasible set, then
\[
    \nabla
    \left(
        \frac12d(\cdot,a)^2
    \right)
    (x^\star)
    =
    -\ell n
    \neq0.
\]
The objective gap along an inward normal displacement is then
$\Theta(u)$ rather than $\Theta(u^2)$. Even if the slow dynamics retain
\[
    u_t
    =
    \Theta\!\left(
        t^{-p/[2(p-1)]}
    \right),
\]
the resulting objective exponent is only
$
  \displaystyle  \frac{p}{2(p-1)}.
$
Centering the objective at $x^\star$ is therefore essential for the
same-policy matching exponent.
\end{remark}

\subsection{The critical cubic endpoint}
\label{subsec:p3-endpoint}

The restriction $p>3$ arises from the stability of the normalized transverse
dynamics. At $p=3$,
\[
    F_3'(C_3)=-1,
\]
so the flat limiting map is marginal rather than strictly contracting.
The $C^1$ perturbation argument used above therefore gives no persistence
conclusion at the cubic endpoint.

In the centered hyperbolic construction, ambient metric corrections enter
the normalized equations at order
$
    \kappa^2u^{2/3},
$
while the boundary-jet correction appearing in
\eqref{eq:normalized-lmo-map-2k} is of order
$
    u.
$
Determining the fate of the marginal multiplier therefore requires the
signed next-order expansion of the full transverse map. The present
analysis does not determine that sign. Consequently, whether the same slow
mechanism persists at $p=3$ remains open within this construction.

\section{Numerical experiments}
\label{sec:numerics}

We complement the analytical results with numerical validation of three
quantitative mechanisms developed in the paper: the sharp hyperbolic rate
under the gap-feedback short-step policy, the curvature-induced distortion
of the logarithm map, and the exact spectral certificate on the
positive-definite cone. The computations directly reproduce the quantities
appearing in the corresponding theoretical results and verify their
predicted scaling, distortion, and spectral behavior across the tested
configurations.

\subsection{Hyperbolic same-policy rate}
\label{subsec:numerics-hyperbolic}

We first consider the hyperbolic hard family of
Section~\ref{sec:lower-bound} and run RFW with the same gap-feedback
short-step policy analyzed in Theorem~\ref{thm:geometry-upper-rate}. The
numerical grid comprises $72$ configurations, obtained from
$p\in\{4,5,6,8\}$, $\rho\in\{0.04,0.08\}$,
$\kappa\in\{0.5,1,2\}$, and three initial transverse states for each
parameter triple. Theorem~\ref{thm:hyperbolic-short-step-lower} predicts
\[
    f_\star(x_t)-f_\star(x^\star)
    =
    \Theta\!\left(t^{-p/(p-1)}\right).
\]
The numerical trajectories reproduce this power law across the full grid.
The maximum absolute deviation between the fitted log--log slope and the
predicted value $-p/(p-1)$ is $2.83\times10^{-3}$. The normalized tail
$t^{p/(p-1)}(f_\star(x_t)-f_\star(x^\star))$ is correspondingly stable,
with a maximum $10$--$90$ percentile spread of $9.66\times10^{-4}$ across
the tested configurations.

The cap in the short-step rule remains inactive throughout all runs, so the
observed asymptotic behavior lies entirely in the interior short-step regime
analyzed in Section~\ref{sec:lower-bound}. Representative trajectories are
shown in Figure~\ref{fig:hyperbolic-rates}. We additionally report the
critical case $p=3$ as an endpoint diagnostic. Its behavior reflects the
loss of strict contraction predicted by the rescaled dynamics at the
critical exponent, whereas the matching-rate result corresponds to the
strictly contractive regime $p>3$ covered by
Theorem~\ref{thm:hyperbolic-short-step-lower}.

\begin{figure}[t]
\centering
\includegraphics[width=.78\linewidth]
{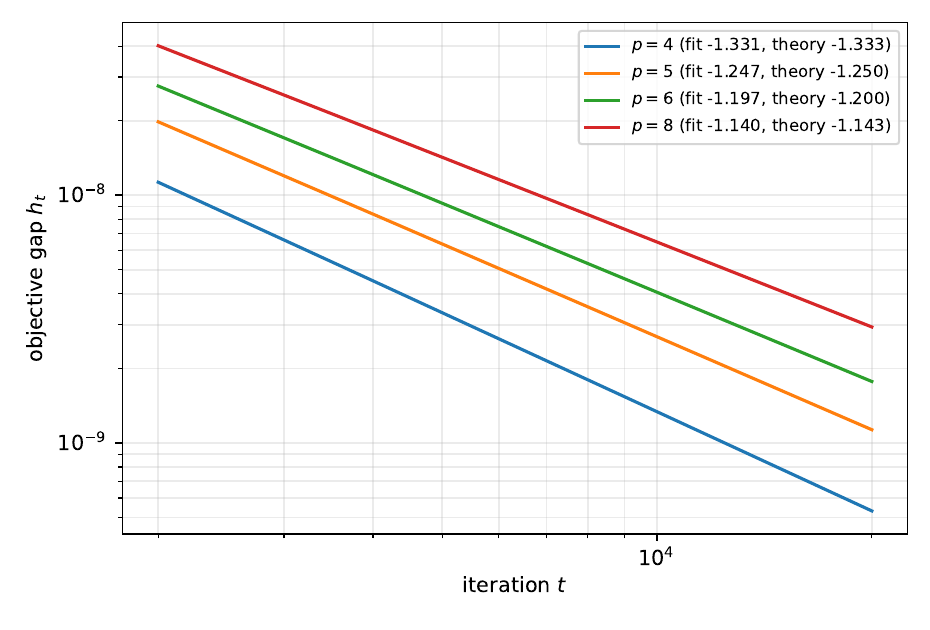}
\caption{Objective gaps for representative trajectories from the
hyperbolic hard family under the gap-feedback short-step policy.
Parentheses report the fitted and predicted log--log slopes. Across the
$72$ tested configurations, the fitted slopes reproduce the predicted
exponent $-p/(p-1)$ with maximum absolute deviation
$2.83\times10^{-3}$.}
\label{fig:hyperbolic-rates}
\end{figure}

\subsection{Jacobi distortion and the Rauch certificate}
\label{subsec:numerics-jacobi}

We next quantify the differential distortion appearing in the exact Jacobi
transfer of Section~\ref{sec:jacobi-scaling}. In constant sectional
curvature $-\kappa^2$, we approximate the logarithm differential by centered
intrinsic finite differences, with perturbations generated through the
exponential map.

The numerical grid comprises $168$ directional configurations:
$\kappa\in\{0.5,1,2\}$, eight values of $\kappa r$ between $0.05$ and $4$,
and seven directions relative to the radial mode. The exact
smallest-singular-value factor from Section~\ref{sec:certificates} is
\[
    \sigma_{\min}\!\left((D\Log_x)_v\right)
    =
    \frac{\kappa r}{\sinh(\kappa r)}.
\]
Across the full grid, the numerically differentiated values reproduce this
factor with a maximum relative error of $1.34\times10^{-11}$. The
directional distortion
$\|(D\Log_x)_v^\ast w\|_v/\|w\|_x$ likewise reproduces its closed-form value
with a maximum relative error of $8.23\times10^{-11}$.

Figure~\ref{fig:jacobi-distortion} displays the resulting directional
structure. The radial mode remains undistorted, the transverse mode follows
the Rauch factor and decreases with $\kappa r$, and intermediate directions
interpolate between these two extremes. The computation therefore confirms
the distinction captured by the LMO-restricted coefficient:
$\Lambda_{\C}^{\LMO}$ depends on the oracle-exposed directions and can be
strictly sharper than a certificate based on the worst tangent direction.

\begin{figure}[t]
\centering
\includegraphics[width=.78\linewidth]
{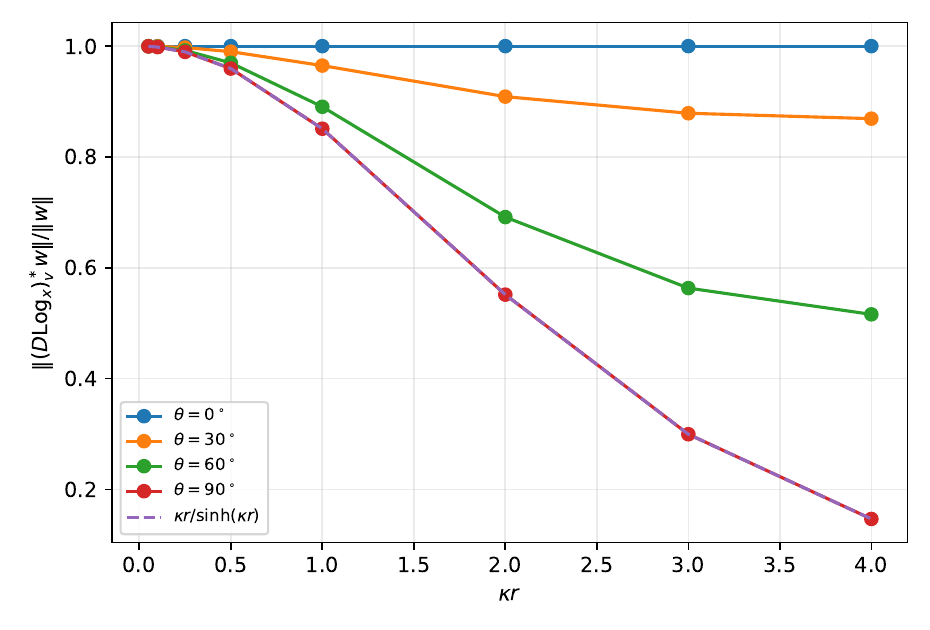}
\caption{Logarithm-map differential distortion in constant negative
curvature. The dashed curve shows the exact transverse factor
$\kappa r/\sinh(\kappa r)$. The radial direction remains undistorted,
while intermediate directions interpolate between the radial and
transverse modes.}
\label{fig:jacobi-distortion}
\end{figure}

\subsection{Affine-invariant SPD spectral certificate}
\label{subsec:numerics-spd}

We finally quantify the logarithm-differential spectrum on $\SPD(n)$
equipped with the affine-invariant Riemannian metric. Let $q$ denote the
condition number of $X^{-1/2}YX^{-1/2}$. Section~\ref{sec:certificates}
gives the exact certificate
\[
    \sigma_{\min}\!\left((D\Log_X)_Y\right)
    =
    \frac{\sqrt q\,\log q}{q-1},
\]
with continuous extension equal to $1$ at $q=1$.

The numerical grid comprises $56$ anisotropic configurations in dimensions
$n\in\{2,3,5,10\}$, spanning seven relative log-spectral spreads and two
common spectral shifts. Across the full grid, the numerically computed
smallest singular values reproduce the exact formula with a maximum
relative error of $1.04\times10^{-8}$. The complete numerical singular
spectra likewise reproduce their predicted values with a maximum relative
error of $1.93\times10^{-8}$.

We additionally consider $16$ isotropic pairs $Y=cX$. For these pairs,
$q=1$, and the logarithm differential has smallest singular value exactly
equal to one. The maximum numerical deviation from this value is
$1.51\times10^{-8}$, even though the affine-invariant distance among the
tested pairs reaches $15.811$. Thus large geodesic separation does not by
itself produce logarithmic distortion on the positive-definite cone.
Relative spectral spread, rather than distance alone, determines the exact
lower differential certificate.

Figure~\ref{fig:spd-spectral} shows the numerical values collapsing onto
the closed-form spectral curve across all tested dimensions.

\begin{figure}[t]
\centering
\includegraphics[width=.78\linewidth]
{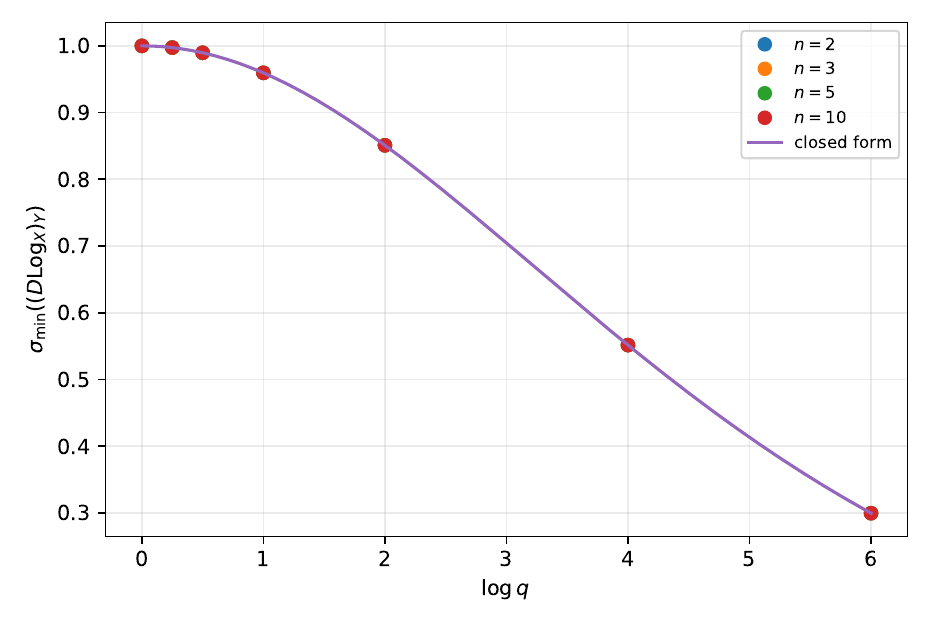}
\caption{Smallest singular value of the affine-invariant SPD logarithm
differential as a function of relative spectral spread. Finite-difference
computations in dimensions $2,3,5,10$ collapse onto the exact spectral
certificate, with maximum relative error $1.04\times10^{-8}$ across the
$56$ anisotropic configurations.}
\label{fig:spd-spectral}
\end{figure}

\begin{table}[t]
\centering
\caption{Numerical agreement with the theoretical predictions. Reported
discrepancies are maxima over the corresponding parameter grids.}
\label{tab:numerical-audit}
\begin{tabular}{llll}
\toprule
Setting & Cases & Quantity tested & Maximum discrepancy \\
\midrule
Hyperbolic same-policy rate
    & 72
    & log--log slope
    & $2.83\times10^{-3}$ \\
Jacobi distortion
    & 168
    & $\sigma_{\min}(D\Log)$
    & $1.34\times10^{-11}$ \\
SPD spectral certificate
    & 56
    & $\sigma_{\min}(D\Log)$
    & $1.04\times10^{-8}$ \\
SPD isotropic case
    & 16
    & $\sigma_{\min}(D\Log)-1$
    & $1.51\times10^{-8}$ \\
\bottomrule
\end{tabular}
\end{table}

Taken together, the computations reproduce the three quantitative
mechanisms underlying the theory. The hyperbolic trajectories attain the
predicted same-policy power law, the numerical logarithm differential
recovers the curvature-dependent Jacobi factor, and the SPD computations
recover the exact dependence on relative spectral spread, including the
undistorted isotropic regime at large geodesic distance. These results
numerically confirm the scaling, distortion, and spectral mechanisms
established in the preceding sections. The $p=3$ diagnostic further
exhibits the critical transition predicted by the rescaled dynamics, while
the matching-rate theorem applies to the strictly contractive regime
$p>3$.

\paragraph{Reproducibility.}
All numerical results reported in this section are reproducible from the
accompanying repository. The repository contains the implementation,
configuration files, numerical outputs, and figure-generation scripts for
the hyperbolic same-policy rate, the Jacobi distortion and Rauch
certificate, and the affine-invariant SPD spectral certificate. It also
provides instructions for reproducing the reported summary statistics and
figures from a clean environment. The complete reproducibility package is
available at
\href{https://github.com/aisia-research-lab/riemannian-pfo-curvature-scaling}
{the accompanying repository}.

\section{Discussion and outlook}
\label{sec:discussion}

This work develops a geometric mechanism governing the complexity of
Riemannian projection-free optimization. The feasible-set modulus quantifies
the boundary geometry available to the linear minimization oracle, while
the differential of the logarithm map transfers this geometry exactly into
Frank--Wolfe scaling through the Jacobi transfer principle. Ambient
curvature and structured spectral geometry provide quantitative
certificates for the resulting differential distortion, and objective-side
error bounds convert the induced scaling inequality into convergence rates.
Together, these components connect intrinsic feasible-set geometry,
Riemannian distortion, and projection-free complexity.

A central feature of the framework is its modularity. The exact transfer
theorem applies directly to the feasible-set modulus, with power-type
geometry providing the specialization that generates an explicit complexity
spectrum. Rauch comparison and the affine-invariant SPD analysis quantify
the ambient differential distortion independently of the subsequent
convergence argument. The framework therefore separates two geometric
effects: the intrinsic strength of the feasible region and the distortion
introduced when this geometry is represented through the Riemannian
logarithm map.

For power-type feasible sets, this mechanism yields
\[
    f(x_t)-f^\star
    =
    O\!\left(t^{-p/(p-1)}\right)
\]
under the gap-feedback short-step policy. The hyperbolic construction of
Section~\ref{sec:lower-bound} attains the matching rate
\[
    f(x_t)-f^\star
    =
    \Theta\!\left(t^{-p/(p-1)}\right)
\]
for every fixed $p>3$ on a genuinely negatively curved family satisfying
the same geometric assumptions. The construction preserves the power-type
feasible-set geometry under the rescaled Klein metric and the attracting
slow dynamics under the same short-step rule used in the upper-bound
analysis. Hence, the exponent $p/(p-1)$ is sharp for this policy over the
geometry-controlled class considered here. In particular, the matching
power law persists under genuine negative curvature and under the same
algorithmic policy.

This same-policy sharpness result points to a broader oracle-complexity
question: characterizing rates that are unavoidable across classes of
Riemannian projection-free algorithms. Such a theory would require
information-based constructions that control adaptive oracle queries and
would complement the dynamical sharpness mechanism established here.

The cubic endpoint $p=3$ presents a distinct dynamical regime. At this
value, the limiting transverse multiplier satisfies
$F_3'(C_3)=-1$, so strict contraction is lost. In the centered hyperbolic
construction, the first ambient metric correction appears at order
$u^{2/3}$, whereas the boundary-jet correction appears at order $u$.
The endpoint is therefore governed by the signed next-order term of the
transverse dynamics. Determining this term would resolve whether the slow
orbit persists with the same exponent, enters a different polynomial
regime, or exhibits a curvature-dependent transition.

The distinction between global and oracle-restricted distortion suggests a
second refinement of the complexity theory. The LMO-restricted quantity
$\Lambda_{\C}^{\LMO}$ probes only tangent directions realized by the linear
minimization oracle, whereas a global smallest-singular-value certificate
controls all tangent directions. The constant-curvature and
positive-definite examples demonstrate that these quantities can encode
substantially different geometric information. Radial directions can
remain undistorted while transverse directions deteriorate, and on
$\SPD(n)$ large geodesic separation can coexist with unit logarithmic
distortion when the relative spectral spread vanishes. Oracle-restricted
geometry can therefore yield sharper conditioning information than a
global manifold certificate.

The framework also extends naturally along several directions. Approximate
linear minimization oracles introduce an additional perturbation into the
geometry-to-scaling transfer, leading to an interaction between oracle
accuracy, feasible-set curvature, and logarithmic distortion. Replacing the
quadratic error bound by broader objective growth conditions should produce
corresponding geometry-dependent complexity regimes. Structured manifolds
beyond the positive-definite cone provide another setting in which explicit
or computable LMO-restricted distortion certificates may be derived.

These results support a geometry--scaling--complexity principle for
Riemannian projection-free optimization: intrinsic feasible-set geometry
generates oracle scaling, ambient differential geometry determines its
conditioning, and objective growth converts the resulting margin into
algorithmic complexity. The exact Jacobi transfer makes the first
connection explicit, the curvature and spectral certificates quantify the
second, and the hyperbolic construction establishes sharpness of the
resulting power-law exponent for the gap-feedback short-step policy.
Oracle-level complexity, critical-endpoint dynamics, approximate oracles,
and structured manifold classes provide natural directions for extending
this principle.

\backmatter

\section*{Declarations}

\begin{itemize}
    \item \textbf{Conflict of interest:} The author declares no competing interests.
    \item \textbf{Funding:} The author received no funding for this work.
    \item \textbf{Data availability:} No external datasets were used in this
    study. All numerical results are generated from synthetic geometric
    instances; the corresponding reproducibility materials are available
    in the repository cited in Section~\ref{sec:numerics}.
\end{itemize}

\bibliography{references}

\end{document}